\documentclass[12pt,reqno]{amsart}
\usepackage{setspace}
\usepackage{amsmath, amssymb,amsthm, amsfonts, latexsym, mathrsfs}
\usepackage{color}
\usepackage{graphicx}
\usepackage[normalem]{ulem}
\usepackage{comment}

\allowdisplaybreaks

\usepackage{etoolbox}
\makeatletter
\let\ams@starttoc\@starttoc
\makeatother
\usepackage[parfill]{parskip}
\makeatletter
\let\@starttoc\ams@starttoc
\patchcmd{\@starttoc}{\makeatletter}{\makeatletter\parskip\z@}{}{}
\makeatother

\newtheorem{thm}{Theorem}[section]
\newtheorem{lemma}[thm]{Lemma}
\newtheorem{prop}[thm]{Proposition}
\newtheorem{cor}[thm]{Corollary}
\newtheorem{df}[thm]{Definition}

\newtheorem{rem}[thm]{Remark}

\newtheorem*{thm*}{Theorem}
\newtheorem*{prop*}{Proposition}
\newtheorem*{lemma*}{Lemma}
\newtheorem*{clai*}{Claim}
\newtheorem*{exam*}{Example}

\newcommand{\tr}{\mbox{tr}}

\renewcommand{\div}{\mbox{div}}

\newcommand{\Ric}{\mbox{Ric}}
\newcommand{\R}{\mathbb R}
\newcommand{\Z}{\mathbb Z}

\numberwithin{equation}{section}
\newcommand{\be}{\begin{equation}}
\newcommand{\ee}{\end{equation}}
\def\p{\partial}

\def\la{\langle}
\def\ra{\rangle}
\def\lf{\left}
\def\ri{\right}
\def\Pi{\displaystyle{\mathbb{II}}}
\def\Ric{\text{\rm Ric}}
\def\e{\epsilon}

\def\a{\alpha}

\def\bg{\bar{g}}
\def\ol{\overline}

\def\sing{\operatorname{sing}}
\def\dist{\operatorname{dist}}
\def\loc{\operatorname{loc}}
\def\spt{\operatorname{spt}}
\def\dvol{dv}

\def\divg{\operatorname{div}}

\def\eps{\epsilon}

\def\madm{\mathfrak{m}{_{ADM}}}
\def\bee{\begin{equation*}}
\def\eee{\end{equation*}}

\def\mby{\mathfrak{m}{_{BY}}}
\def\mhawk{\mathfrak{m}{_{H}}}

\def\wt{\widetilde}
\def\wh{\widehat}

\def\Sh{\Sigma_{H}}
\def\FF{\mathcal{F}}
\def\LL{\Lambda}
\def\F{\mathring{\mathcal{F}}}
\def\L{\mathring{\Lambda}}
\def\Fcirc{\mathring{\FF}}
\def\Lcirc{\mathring{\LL}}

\def\C{\mathcal{C}}

\def\CgSM{\C_g(\Sigma, M)}

\def\exp{\mathrm{exp}}

\def\sff{\mathrm{I\!I}}

\newcounter{mnotecount}[section]

\title
{Capacity, quasi-local mass, and singular fill-ins}

\author{Christos Mantoulidis$^1$}
\address{Department of Mathematics\\Massachusetts Institute of Technology\\Cambridge, MA 02139}
\email{c.mantoulidis@mit.edu}
\thanks{{$^1$Research partially supported by NSF Grant DMS-1905165, and completed at Institut Mittag-Leffler's “General relativity, geometry and analysis” workshop, supported from grant 2016-06596 of the Swedish Research Council.}}

\author{Pengzi Miao$^2$}
\address{Department of Mathematics\\University of Miami\\Coral Gables, FL 33146, USA}
\email{pengzim@math.miami.edu}
\thanks{{$^2$Research partially supported by NSF Grant DMS-1906423.}}

\author{Luen-Fai Tam$^3$}
\address{The Institute of Mathematical Sciences and Department of Mathematics\\The Chinese University of Hong Kong\\Shatin, Hong Kong, China}
\email{lftam@math.cuhk.edu.hk}
\thanks{{$^3$}Research partially supported by Hong Kong RGC General Research Fund \#CUHK 14301517}

\subjclass[2010]{Primary 53C20; Secondary 83C99}

\date{\today}

\begin{document}
	
\begin{abstract}
	 We derive new inequalities between the boundary capacity of an asymptotically flat 3-manifold with nonnegative scalar curvature and boundary quantities that relate to quasi-local mass; one relates to Brown--York mass and the other is new. We argue by recasting the setup to the study of mean-convex fill-ins with nonnegative scalar curvature and, in the process, we consider fill-ins with singular metrics, which may have independent interest. Among other things, our work yields new variational characterizations of Riemannian Schwarzschild manifolds and new comparison results for surfaces in them.
\end{abstract}

\keywords{capacity, quasi-local mass, scalar curvature}
\maketitle
\markboth{Christos Mantoulidis, Pengzi Miao and Luen-Fai Tam}{Capacity, quasi-local mass, singular fill-ins}

\section{Introduction} \label{s-intro}

The Riemannian (``time-symmetric'') case of the Penrose Conjecture \cite{Penrose73} was settled in different forms by Huisken--Ilmanen \cite{HuiskenIlmanen01} and by Bray \cite{Bray01}. The latter, specifically, proved the following relation between the \emph{ADM mass} \cite{ADM59} of an asymptotically flat manifold and the \emph{Hawking quasi-local mass} \cite{Hawking68} of its apparent horizon:

\begin{thm*}[Riemannian Penrose Inequality, {\cite[Theorem 1]{Bray01}}]
	Let $(M^3, g)$ be complete, asymptotically flat, with nonnegative scalar curvature and an outer minimizing horizon $\Sigma$ (possibly disconnected) with area $|\Sigma|$ and ADM mass $\madm(M, g)$. Then
	\begin{equation} \label{e-RPI}
		\sqrt{\frac{|\Sigma|}{16\pi}} \leq \madm(M, g).
	\end{equation}
	Equality holds if and only if, outside $\Sigma$, $(M, g)$ is isometric to a Riemannian half-Schwarzschild manifold.
\end{thm*}
Recall that in the Riemannian setting horizons are closed minimal surfaces, and their Hawking mass is the left hand side of \eqref{e-RPI}.

Harmonic potentials and the notion of \emph{boundary capacity} played an important role in Bray's work by virtue of the following relation between ADM mass and boundary capacity:

\begin{thm*}[Mass--Capacity inequality, {\cite[Theorem 9]{Bray01}}]
	Let $(M^3, g)$ be complete, asymptotically flat, with nonnegative scalar curvature and with horizon boundary $\Sigma$. Then
	\be \label{e-m-cap}
		\CgSM \leq \madm(M, g),
	\ee
	with $\CgSM$ as in Definition \ref{d-capacity}. Equality holds if and only if $(M, g)$ is isometric to a Riemannian half-Schwarzschild manifold.
\end{thm*}

\begin{df}[Capacity] \label{d-capacity}
	Let $(M^3, g)$ be a complete, asymptotically flat Riemannian manifold with compact boundary $\Sigma$. Let $\phi : M \to \R$ (the ``boundary capacity potential'') be such that
	\begin{equation} \label{e-capacity-pde}
		\begin{array}{rl}
			\Delta_g \phi = 0 & \mathrm{on} \ M, \\
			\phi = 1 & \mathrm{at} \ \Sigma, \\
			\phi \rightarrow 0 & \mathrm{as} \ x \rightarrow \infty .
		\end{array}
	\end{equation}
	The boundary capacity of $\Sigma \subset (M, g)$ is:
	\be \label{e-capacity-meaning}
	\CgSM := \frac{1}{4 \pi} \int_M | \nabla \phi |^2 \, dv = \frac{1}{4\pi} \int_{\Sigma} \tfrac{\p}{\p \eta} \phi \, d\sigma,
	\ee
	where $\eta$ is the unit normal pointing outside of $M$. (The second equality in \eqref{e-capacity-meaning} follows from  Lemma \ref{l-phi}.)
\end{df}

Before stating the main results in this paper, we first state a corollary of Theorem \ref{t-LC-1} that needs no new definitions:

\begin{cor} \label{c-capacity-Gauss}
	Let $(M^3, g)$ be complete, asymptotically flat, with nonnegative scalar curvature and horizon boundary $\Sigma$ that consists of  spheres with positive Gauss curvature. Then:
	\begin{equation} \label{e-capacity-spheres-Gauss}
		\CgSM \leq \frac{1}{16\pi} \int_{\Sigma} H_{g_0} \, d\sigma
	\end{equation}
	where $d\sigma$ is the induced area form and $H_{g_0}$ is the mean curvature of the isometric embedding of each boundary sphere of $\Sigma$ in $\R^3$. Equality holds if and only if $(M, g)$ is isometric to a Riemannian half-Schwarzschild manifold with mass $\CgSM$.
\end{cor}

\begin{rem}[Mean curvature convention] \label{r-mean-curvature}
	Suppose $\Sigma$ is a hypersurface in a Riemannian manifold. Its vector-valued second fundamental form is defined as $\vec \sff(X,Y)= (\nabla_XY)^\perp$ for $X, Y$ tangent to $\Sigma$, where $Z^\perp$ is the normal projection of $Z$. Its mean curvature vector $\vec H $ is the trace of $\vec \sff$. If $\nu$ is a unit normal to $\Sigma$, then the mean curvature with respect to $\nu$ is given by $H=H(\nu)=g(-\vec H,\nu)$. Unless stated otherwise, for hypersurfaces $\Sigma$ inside an asymptotically flat manifold (with one end---see Definition \ref{d-AF}), $\nu$ will be  taken to be the unit normal pointing to the distinguished end. When there is no asymptotically flat manifold, and $\Sigma$ is the boundary of a compact manifold $(\Omega, g)$  (e.g., Sections \ref{s-singular}, \ref{s-isolated-singularities}, and in defining $\LL(\Sigma, \gamma)$, $\Lcirc(\Sigma, \gamma)$ in Definition \ref{d-ff-ll}), $\nu$ will be taken to be the unit normal pointing outside of $\Omega$.
\end{rem}

The isometric embedding used to define $H_{g_0}$ exists since, by \cite{Nirenberg53, Pogorelov64}, all Riemannian two-spheres with positive Gauss curvature can be uniquely isometrically embedded into $(\R^3, g_0)$. If $\Sigma$ does not consist of spheres with positive Gauss curvature, then one may not be able to isometrically embed $\Sigma$ in $\R^3$, so the right hand side of \eqref{e-capacity-spheres-Gauss} is not well-defined. Hence, following \cite{M-M}, we consider the $\Lambda$-invariant instead:

\begin{df} \label{d-ff-ll} \label{d-f-l}
	Let $\Sigma$ be a closed, orientable surface (possibly disconnected) endowed with a Riemannian metric $\gamma$. Denote as $\FF(\Sigma, \gamma)$ the set of compact, connected Riemannian manifolds $(\Omega^3, g)$ with boundary such that:
	\begin{enumerate}
		\item[(i)] $\p \Omega$, with the induced metric, is isometric to $(\Sigma, \gamma)$,
		\item[(ii)] the mean curvature vector of $\p \Omega$ points strictly inward, and
		\item[(iii)] $R(g) \geq 0$, where $R(g)$ is the scalar curvature of $g$.
	\end{enumerate}
	Then, define
	\[ \LL(\Sigma, \gamma) := \sup \left\{ \frac{1}{8\pi} \int_{\p \Omega} H_g \, d\sigma \; | \; (\Omega, g) \in \FF(\Sigma, \gamma) \right\}. \]
	Let $\Fcirc(\Sigma,\gamma)$ be the set of compact, connected 3-dimensional $(\Omega, g)$ with $\p \Omega = \Sigma_O \sqcup \Sigma_H$ such that (i)-(ii) hold with $\Sigma_O$ in place of $\p \Omega$, (iii), and
	\begin{enumerate}
		\item [(iv)] $\Sigma_H$, if nonempty, is an arbitrary minimal surface (possibly disconnected).
	\end{enumerate}
	Correspondingly, define
	\[ \Lcirc(\Sigma,  \gamma) :=\sup \left\{ \frac{1}{8\pi} \int_{\p \Omega} H_g \, d\sigma \; | \; (\Omega, g) \in \Fcirc(\Sigma, \gamma) \right\}. \]
	Above, $H_g$ denotes the mean curvature of $\p \Omega$ with respect to the  unit normal pointing outside of $\Omega$. (Note: $\Fcirc$, $\Lcirc$ were  defined  for connected $\Sigma$ in \cite{M-M}, but the same results hold for general $\Sigma$; see Appendix \ref{s-f-l-disconnected}.)
\end{df}

In Definition \ref{d-f-l}, $\FF(\Sigma, \gamma) \subset \Fcirc(\Sigma, \gamma)$, so $\LL(\Sigma, \gamma) \leq \Lcirc(\Sigma, \gamma)$. Nonetheless,
 $\LL(\Sigma,\gamma)= \Lcirc(\Sigma, \gamma)$ by \cite[Proposition 5.1]{M-M} and Proposition \ref{p-equality-l-ll}. It was shown in \cite[Theorem 1.3]{M-M} (see also \cite{Lu-isometric-embedding}) that $\LL(\Sigma, \gamma) < \infty$ for finite unions of topological spheres. By convention, $\LL(\Sigma, \gamma) := -\infty$ when $\FF(\Sigma, \gamma) = \emptyset$; likewise for $\Lcirc$, $\Fcirc$. It was observed in \cite[Remark 1.5]{M-M} that $\FF(\Sigma, \gamma) \neq \emptyset$ if $\Sigma$ consists of spheres satisfying the spectral condition $\lambda_1(-\Delta_\gamma + K_\gamma) > 0$, where $K_\gamma$ denotes the Gauss curvature.

Our first main theorem is the following. All asymptotically flat manifolds in our work will be asymptotically flat of order $\tau > \tfrac12$ in the sense of Definition \ref{d-AF}. Recall Remark \ref{r-mean-curvature} for mean curvature conventions in asymptotically flat manifolds.

\begin{thm} \label{t-LC-1}
	Let $(M^3, g)$ be complete, asymptotically flat, with compact, orientable boundary $\Sigma$ with induced metric $\gamma$ and with nonnegative scalar curvature. Let $\phi$ be the boundary capacity potential for $\CgSM$. Suppose that:
	\begin{equation} \label{e-LC-1-boundary-assumption}
		H_{g} < 4 |\nabla \phi| \text{ on } \Sigma,
	\end{equation}
	where $H_{g}$ is the mean curvature of $\Sigma$. Then $\FF(\Sigma, \gamma) \neq \emptyset$ and:
	\begin{enumerate}
		\item We have the inequality
		\begin{equation} \label{e-LC-1}
			2 \, \CgSM \le \LL(\Sigma,\gamma) + \frac1{8\pi} \int_\Sigma H_{g}\,d\sigma.
		\end{equation}
		\item If equality holds in \eqref{e-LC-1}, then $\Sigma$ is connected, $(M,g)$ is scalar flat, and $(M,g)$ is conformally equivalent to a mean-convex handlebody whose boundary is isometric to $(\Sigma, \gamma)$ and whose interior metric is flat. (Recall: handlebodies are compact manifolds diffeomorphic to the boundary connect sum of finitely many solid tori ($B^2 \times \mathbb{S}^1$), or $B^3$. See \cite[Proposition 1]{MeeksSimonYau82}.)
		\item If $\Sigma$ consists of topological spheres (e.g., $M$ is simply connected) with constant nonnegative mean curvature, then equality holds in \eqref{e-LC-1} if and only if $(M,g)$ is isometric to the exterior of a rotationally symmetric sphere in a Riemannian Schwarzschild manifold.
	\end{enumerate}
\end{thm}

We note that \eqref{e-LC-1-boundary-assumption} {obviously} allows {any} boundary $\Sigma$ {with $H_g \leq 0$, but it also allows important cases where $\Sigma$ has $H_g > 0$; for} instance, the theorem applies to every coordinate sphere in a Riemannian Schwarzschild manifold {(}see Section \ref{s-schwarzschild-computations}{), as well as to small perturbations.}

Regarding the capacity $\CgSM$ and the boundary mean curvature $H_g$, H. Bray and the second author proved the following:

\begin{thm*}[{\cite[Theorem 1]{BrayMiao07}}]\label{t-BM}
	Let $(M^3, g)$ be complete, asymptotically flat, with nonnegative scalar curvature and connected boundary $\Sigma$. Assume $M$ is diffeomorphic to $\R^3 \setminus \Omega$, where $\Omega$ is a bounded domain with connected boundary. Then
	\begin{equation} \label{e-braymiao-1}
		\CgSM \leq \sqrt{\frac{|\Sigma|}{16\pi}} \left( 1 + \sqrt{\frac{1}{16\pi} \int_\Sigma H_g^2 \, d\sigma} \right),
	\end{equation}
	where $|\Sigma|$ and $H_g$ are the area and mean curvature of $\Sigma$. Furthermore, equality holds if and only if $(M, g)$ is isometric to the exterior region of a rotationally symmetric sphere outside the horizon in a Riemannian Schwarzschild manifold.
\end{thm*}	
If $\Sigma$ is minimal, then the right side of \eqref{e-braymiao-1} reduces to the Hawking mass of $\Sigma$. It is worth comparing \eqref{e-braymiao-1} and  \eqref{e-LC-1}:
\begin{enumerate}
	\item[(i)] The inequality in \eqref{e-LC-1} becomes an identity on every coordinate sphere
	  in the (doubled) Riemannian Schwarzschild manifold,
	  while \eqref{e-braymiao-1} becomes an identity only on spheres that are in the exterior region of the Riemannian Schwarzschild manifold.
	\item[(ii)] In  \eqref{e-braymiao-1}, it is assumed that $\Sigma$ is connected because the proof   in \cite{BrayMiao07} made use of weak solutions to the inverse mean curvature flow developed by Huisken--Ilmanen \cite{HuiskenIlmanen01}. However, Theorem \ref{t-LC-1} imposes no topological assumptions on $(M, g)$ and	$\Sigma$ is allowed to be disconnected.
	\item[(iii)] In the case that both estimates \eqref{e-LC-1} and \eqref{e-braymiao-1} apply, \eqref{e-LC-1} is a better estimate when $ \Sigma $ is a round sphere and $ H_g $ is not a constant, by H\"{o}lder's inequality. On the other hand, if $H_g=0$ and $\Sigma$ has positive Gauss curvature, \eqref{e-braymiao-1} gives better estimate by the Minkowski inequality in $ \R^3$.
\end{enumerate}

Theorem \ref{t-LC-1} has a corollary that relates to Szeg\"o's theorem, which we recall here (see, e.g., \cite[section 3.4]{PolyaSzego51}): if $\Sigma \subset \R^3$ is a closed surface bounding a convex domain $\Omega$, then
	\begin{equation} \label{e-cap-upper-r3}
		\mathcal{C}_{g_0}(\Sigma, \R^3 \setminus \Omega) \leq \frac{1}{8\pi} \int_\Sigma H_{g_0} \, d \sigma,
	\end{equation}
	with equality if and only if $\Sigma$ is a round sphere. We obtain:

\begin{cor} \label{c-LC-1-vs-szego}
	Let $(M^3, g)$ be complete, asymptotically flat, with nonnegative scalar curvature, and no boundary. Suppose $\Omega \subset M$ is a bounded domain with boundary $\Sigma$ consisting of spheres with positive Gauss curvature, and mean curvature satisfying $0 < H_g < 4 |\nabla \phi|$, with $\phi$ as in \eqref{e-capacity-pde}. Then:
	\begin{equation} \label{e-cap-upper-L-Gauss}
		\mathcal{C}_{g}(\Sigma, M \setminus \Omega) \leq \frac{1}{8\pi} \int_\Sigma H_{g_0} \, d \sigma,
	\end{equation}
	where $H_{g_0}$ is the mean curvature of the isometric embedding of each component of $\Sigma$ in $(\R^3, g_0)$. Equality holds in \eqref{e-cap-upper-L-Gauss} if and only if $\Sigma$ is a round sphere and $(M, g)$ is isometric to $(\R^3, g_0)$.
\end{cor}

This corollary essentially replaces $\R^3$ in Szeg\"o's theorem with an asymptotically flat manifold, and the convexity assumption on $\Sigma$ is replaced by a mean curvature condition. Without the positive Gauss curvature assumption on $\Sigma$, one still has $\mathcal{C}_g(\Sigma, M \setminus \Omega) \leq \LL(\Sigma, \gamma)$, with equality corresponding to the rigidity case of Theorem \ref{t-LC-1} and the rigidity of $\frac{1}{8\pi} \int_\Sigma H_g \, d\sigma = \LL(\Sigma, \gamma)$ \cite[Theorem 1.4]{M-M}.

Corollaries \ref{c-capacity-Gauss} and \ref{c-LC-1-vs-szego} are a special case of Theorem \ref{t-LC-1}, in view of the following result of  Y.-G. Shi and the third author \cite{ShiTam02}, which proves that
\begin{equation} \label{e-ll-h-gauss}
	\LL(\Sigma, \gamma) = \frac{1}{8\pi} \int_\Sigma H_{g_0} \, d\sigma
\end{equation}
for spheres $(\Sigma, \gamma)$ of positive Gauss curvature whose isometric embedding into $(\R^3, g_0)$ has mean curvature $H_{g_0}$:

\begin{thm*}[{\cite[Theorem 1]{ShiTam02}}]
	Let $(\Omega^3, g)$ be a compact, connected Riemannian manifold with nonnegative scalar curvature, and with compact mean-convex boundary $\Sigma$, which consists of spheres with positive Gauss curvature. Then,
	\begin{equation} \label{e-shi-tam}
		\int_{\Sigma_\ell} H_g \, d\sigma \leq \int_{\Sigma_\ell} H_{g_0} \, d\sigma
	\end{equation}
	for each component $\Sigma_\ell \subset \Sigma$, $\ell = 1, \ldots, k$. Moreover, equality holds for some $\ell = 1, \ldots, k$ if and only if $\p \Omega$ has only one component and $(\Omega, g)$ is isometric to a domain in $\R^3$.
\end{thm*}

Note that, by a slight abuse of notation, the right hand side of \eqref{e-capacity-spheres-Gauss} in Corollary \ref{c-capacity-Gauss} equals one-half times the \emph{Brown--York quasi-local mass} \cite{BrownYork1, BrownYork2} (another  important notion of mass that we use in our paper) evaluated \emph{on} the horizon boundary. Recall its definition:

\begin{df}[Brown--York mass] \label{d-mby}
	Let $(\Omega^3, g)$ be a compact, connected  Riemannian manifold with compact boundary $\Sigma$ that consists of spheres with positive Gauss curvature. The Brown--York mass of $(\Sigma;\Omega,  g)$ is defined to be:
	\begin{equation} \label{e-mby}
		\mby(\Sigma;\Omega, g) := \frac{1}{8\pi}   \int_{\Sigma} (H_{g_0}  - H_g) \, d\sigma,
	\end{equation}
	where $d\sigma$ is the area form of the metric $\gamma$ induced on $\Sigma$, $H_g$ is the mean curvature of $\Sigma \subset (\Omega, g)$ computed with respect to the unit normal pointing outside of $\Omega$, and $H_{g_0}$ is the mean curvature of the isometric embedding of each component of $(\Sigma , \gamma)$ into $(\R^3, g_0)$.
\end{df}

Given the {aforementioned} results, one wonders whether the capacity and the various notions of mass (ADM, Hawking, Brown--York) fit into one picture. We recall: An inequality between Brown--York mass and Hawking mass has been established {in} \cite{ShiTam07,Miao09}. (See also \cite{M-M-CMSA}).
On the other hand, if $\Sigma$ are large coordinate spheres bounding compact domains in asymptotically flat manifolds, the Brown--York mass will converge to the ADM mass by \cite{FanShiTam09}. 	
{On the other hand, there} can be no a priori inequality between the Brown--York mass of $\Sigma$ and ADM mass. This can be seen as a consequence of the work of R. Schoen and the first author \cite{MS15} on the \emph{Bartnik quasi-local mass}; see Section \ref{s-examples-inequality-1}.

{Thus, while there is no a priori relation between the right hand sides of \eqref{e-m-cap}, \eqref{e-capacity-spheres-Gauss}, 
our second main theorem in this paper succeeds in giving a ``localization'' of Bray's capacity bound, \eqref{e-m-cap}, thus relating boundary capacity to Brown--York mass of far outlying equipotential spheres. For simplicity, we first present this localization, which is of independent interest, in the form of a corollary to the main theorem:

\begin{cor} \label{c-localization-bray}
	Let $(M^3,g)$ be complete, asymptotically flat, orientable, with nonnegative scalar curvature, and compact boundary $\Sigma$ with nonpositive mean curvature. Let $\phi$ be the boundary capacity potential in \eqref{e-capacity-pde}, and
	\[ u := \tfrac12 (2-\phi), \; \Sigma_c := \{ u = c \}. \]
	There exists $c_0 \in (\tfrac12, 1)$ so that, for all $c \in [c_0, 1)$, $\Sigma_c$ is an embedded sphere with positive Gauss and mean curvatures. If $\Omega_c$ is the \emph{compact} domain bounded by $\Sigma$ and $\Sigma_c$, then
	\begin{equation} \label{e-localization-bray}
		c^{-1} \, \CgSM \leq \mby(\Sigma_c; \Omega_c, g)
	\end{equation}
	for any $c \in [c_0, 1)$. Equality holds for some $c$ if and only if $(M, g)$ is isometric to the exterior region of a Riemannian Schwarzschild with horizon $\Sigma$, in which case equality holds for all $c$.
\end{cor}

The right hand side of \eqref{e-localization-bray} is the Brown--York mass of $\Sigma_c$ and, by \cite{ShiWangWu09,FanShiTam09}, it converges to $\madm(M, g)$ as $c \to 1$. This recovers Bray's capacity inequality \eqref{e-m-cap}.

Our second theorem, which implies Corollary \ref{c-localization-bray}, is:
}


\begin{thm} \label{t-LC-2}
	Let $(M^3,g)$ be complete, asymptotically flat, orientable, with nonnegative scalar curvature, and compact boundary $\Sigma$ with nonpositive mean curvature. Let $\phi$ be the boundary capacity potential in \eqref{e-capacity-pde}, and
	\[ u := \tfrac12 (2-\phi), \; \Sigma_c := \{ u = c \}. \]
	Let $c \in (\tfrac12, 1)$ be a regular value of $u$ so that:
	\begin{equation} \label{e-LC-2-boundary-assumption}
		H_{g} > - 4 |\nabla \log u| \text{ on } \Sigma_c,
	\end{equation}
	where $H_{g}$ is the mean curvature of $\Sigma_c$. Then
	\begin{equation} \label{e-LC-2}
		c^{-1} \, \CgSM \leq \LL(\Sigma_c, \gamma_c) - \frac{1}{8\pi} \int_{\Sigma_c} H_{g} \, d\sigma_c;
	\end{equation}
	here, $\sigma_c$ is the area form of the metric $\gamma_c$ induced on $\Sigma_c$.
	
	If $\Sigma_c$ is a sphere with positive Gauss and mean curvatures, then equality holds in \eqref{e-LC-2} if and only if $(M, g)$ is isometric to the exterior region of a Riemannian Schwarzschild manifold with horizon $\Sigma$.
\end{thm}

We also wish to point out that estimates on suitable equipotential surfaces in static asymptotically flat manifolds were given by Agostiniani--Mazzieri \cite{AgostinianiMazzieri17}.


We conclude the introduction by listing some further results we obtained on the boundary capacity of sets in the Schwarzschild manifolds $(M_m, g_m)$ of mass $m > 0$:
\[ M_m := \R^3 \setminus \{0\}, \; g_m := \left( 1 + \frac{m}{2r} \right)^4 g_0, \]
where $g_0$ denotes the flat Euclidean metric on $\R^3$. They are related to some well-known results in $\R^3$.

Before we state the next result, let us fix some notation. In the following, if $\Omega$ is a domain in $\R^3$ containing the origin with $\p\Omega=\Sigma$, then $M_m^\Sigma$ denotes the set $M_m\setminus\Omega$ in the Schwarzschild manifold $(M_m,g_m)$.

We will generalize the following classical  Poincar\'e--Faber--Szeg\"o capacity inequality on $\R^3$ (see, e.g., \cite[p.17]{PolyaSzego51}). If $\Sigma$ is the boundary of a compact domain $\Omega$, then
	\be \label{e-szego}
		\mathcal{C}_{g_0}(\Sigma, \R^3 \setminus \Omega) \ge \left( \frac{3 |\Omega|}{4\pi} \right)^{1/3},
	\ee
where $|\Omega|$ is the volume of $\Omega$, 	with equality if and only if $\Omega$ is a geodesic ball. In the case of Schwarzschild manifolds, we have:

\begin{thm} \label{t-szego-schwarzschild}
	 Let $\Sigma$ be a closed surface bounding a domain with the horizon $\Sh$ in a Schwarzschild manifold $(M_m, g_m)$ with mass $ m > 0$. Let $\Sigma^*$ be the unique rotationally symmetric sphere in $(M_m , g_m)$ which encloses a domain with $ \Sh$ with the same (signed) volume as $\Sigma$. Then
	\be \label{e-szego-schwarzschild}
		\mathcal{C}_{g_m}(\Sigma, M_m^\Sigma) \ge \mathcal{C}_{g_m}(\Sigma^*, M_m^{\Sigma^*});
	\ee
	here, $M_m^\Sigma \subset M_m$ is exterior of $\Sigma$, and similarly for $M_m^{\Sigma^*}$, $\Sigma^*$. Equality holds in \eqref{e-szego-schwarzschild} if and only if $\Sigma = \Sigma^*$.
\end{thm}

Using this together with Theorems \ref{t-LC-1} and \ref{t-szego-schwarzschild}, one can obtain some new comparisons between the Brown--York masses of certain domains in Schwarzschild and their rotationally symmetric counterparts. See Corollaries \ref{c-schwarzschild-comparison} and  \ref{c-schwarzschild-mby-comparison}  for more details.

\begin{rem} \label{r-capacity-rotsym}
	By examining its proof in Section \ref{s-schwarzschild}, one can see that Theorem \ref{t-szego-schwarzschild} remains true for rotationally symmetric asymptotically flat manifolds that can be foliated by isoperimetric hypersurfaces. See \cite{BrayMorgan02,ChodoshEichmairShiYu} for some sufficient conditions on the existence of such a foliation.
\end{rem}

The organization of the paper is as follows. In Section \ref{s-hf-asymptotics}, we   prove a special case of Corollary \ref{c-capacity-Gauss} to motivate our approach. This leads us to consider   fill-ins similar to those in $\FF(\Sigma,\gamma)$ but whose metric need not be smooth. Singular fill-ins will be discussed in detail in Section \ref{s-singular}. In some cases, one can consider singular fill-ins with less stringent conditions. This will be discussed in Section \ref{s-isolated-singularities}. In Section \ref{s-LC}, we prove Theorems \ref{t-LC-1}, \ref{t-LC-2} and their corollaries. In Section \ref{s-schwarzschild}, we prove Theorem \ref{t-szego-schwarzschild} for Schwarzschild manifolds and some interesting corollaries. In Section \ref{s-examples-inequality-1} we give   some useful examples that help illustrate our results. To obtain the rigidity parts of Theorems \ref{t-LC-1} and \ref{t-LC-2}, we need a result in function theory, which is given in Section \ref{s-isometry}. {We believe this result is of independent interest in the study of overdetermined elliptic PDE.}

\section{Motivating outline}
\label{s-hf-asymptotics}

As motivation for the tools developed later in this paper, 
{let us give an outline of the}
proof of Corollary \ref{c-capacity-Gauss} to Theorem \ref{t-LC-1}.

Consider the boundary capacity potential $\phi : M \to (0, 1)$ from \eqref{e-capacity-pde} and, then, consider the conformal compactification
\begin{equation} \label{e-g-phi-hf}
	g_{\phi} := \phi^4 g
\end{equation}
Let $(\Omega^3, \bg_\phi)$ denote the metric completion of $(M^3, g_\phi)$.

\begin{clai*}
	 {$\Omega$ is the compactification of the asymptotic end of $M$ to a point, and $\bar g_\phi$ is a Riemannian metric with regularity $W^{1,p}$, $p > 6$, which is (locally) smooth away from an interior point singularity.}
\end{clai*}


{This is shown in Section \ref{s-LC}. The presence and effect of singularities on the arguments we are about to explain is an issue we will have to understand in detail, but (for the purposes of this motivating discussion) let us see how to proceed assuming that $(\Omega, \bar g_\phi)$ is smooth.} The scalar curvature of $R(\bg_\phi)$ is:
\be \label{e-gphi-R}
	R(\bg_\phi) = 8 \, \phi^{-5} (-\Delta_g \phi + \tfrac18 R(g) \phi) = \phi^{-4} R(g) \geq 0.
\ee
By \eqref{e-LC-1-boundary-assumption}, the mean curvature of $\Sigma$ with respect to $\bg_\phi$ is
\be \label{e-gphi-H}
	H_{\bg_\phi} = 4 \tfrac{\p}{\p \nu} \phi > 0,
\ee
where $\nu$ is the unit normal of $\Sigma \subset (M, g)$ pointing outside of $M$.

Applying \cite[Theorem 1]{ShiTam02} to $(\Omega, \bar g_\phi)$, together with \eqref{e-gphi-H}:
\[ 4 \int_{\Sigma_\ell} \tfrac{\p}{\p \nu} \phi \, d\sigma_\ell = \int_{\Sigma_\ell} H_{g_\phi} \, d\sigma_\ell \leq \int_{\Sigma_\ell} H_{g_0} \, d\sigma_\ell \]
for $\ell = 1, \ldots, k$, where $H_{g_0}$ denotes the mean curvature of the isometric embedding of $\Sigma_\ell$, with its induced metric, into $(\R^3, g_0)$. Summing over $\ell = 1, \ldots, k$, integrating by parts using \eqref{e-capacity-pde} and the decay from Lemma \ref{l-phi}, we get
\[ 16\pi \, \CgSM = \int_{\Sigma} \tfrac{\p}{\p \nu} \phi \, d\sigma \leq \sum_{\ell=1}^k \int_{\Sigma_\ell} H_{g_0} \, d\sigma_\ell, \]
which is precisely \eqref{e-capacity-spheres-Gauss}.

Suppose, now, that equality holds for \eqref{e-capacity-spheres-Gauss}. We're in the rigidity case of \cite[Theorem 1]{ShiTam02} above. This already implies that $k=1$ and that $(\Omega, \bar g_\phi)$ is isometric to a connected domain $(\Omega, g_0)$ in $\R^3$. Moreover, we may assume that $\Omega$ contains a point $p$ that corresponds to $\infty$ in $(M, g)$. By Theorem \ref{t-isometry}, one can show $(\Omega,\bar g_\phi)$ is isometric to a Euclidean ball with center at $p$. This will imply that $(M,g)$ is a Riemannian Schwarzschild manifold outside the horizon. The details will be given in the proof of the general case
{in Section \ref{s-LC}.}

\section{Singular fill-ins} \label{s-singular}

Motivated by our discussion 
in Section \ref{s-hf-asymptotics}, one would like to study metrics such as $g_\phi$ from \eqref{e-g-phi-hf}, which may be singular at a point. For later {purposes}, we also will consider metrics which may not be smooth on a more general subset. Hence we consider the following:

\begin{df}[Singular fill-ins] \label{d-fsing}
	Let $\Sigma$ be a closed, orientable manifold of dimension $n-1$ with $n\ge 3$  (possibly disconnected) endowed with a smooth Riemannian metric $\gamma$. A smooth compact  manifold  $\Omega$ with boundary, endowed with {an $L^\infty$} Riemannian metric $g$ is said to be a (possibly) singular fill-in of $(\Sigma,\gamma)$ if
	\begin{enumerate}
		\item[(i)] $g$ is $C^\infty$ locally away from a compact subset $\sing g \subset \Omega \setminus \p \Omega$;
		\item[(ii)] $\p \Omega$, with the induced metric, is isometric to $(\Sigma, \gamma)$;
		\item[(iii)]  the mean curvature of $\p \Omega$ is positive; and
		\item[(iv)] $R(g) \geq 0$ on $\Omega \setminus \sing g$.
	\end{enumerate}
\end{df}

{Recall that an $L^\infty$ Riemannian metric $g$ on a smooth compact manifold(-with-boundary) is a measurable section $g$ of the space of positive definite symmetric 2-tensors, so that $c^{-1} g_0 \leq g \leq c g_0$, $g$-a.e., for a smooth background metric $g_0$  and a $c > 1$.}

As in the case $n=3$, if $g$ is smooth, then we say that $(\Omega,g)$ is a (smooth) fill-in of $(\Sigma,\gamma)$.

\begin{rem} \label{r-mean-curvature-singular}
	In this section the mean curvature of $\p \Omega$ is computed as in Remark \ref{r-mean-curvature}, with $\nu$ being the unit normal pointing outside of $\Omega$.
\end{rem}

A singular fill-in $(\Omega,g)$  of $(\Sigma,\gamma)$  is said to   satisfy condition $(\mathbf{C})$  if
  $\sing g$ is a disjoint union of compact sets $P$, $Q$ that satisfy:
\begin{itemize}
\item[({\bf a})] $P$ is  a smoothly embedded two-sided compact hypersurface without boundary,
	\begin{enumerate}
		\item[(i)] near $P$, $g$ can be expressed as
			\[ g=dt^2+g_{\pm}(t,z) \]
			for smooth coordinates $(t, z) \in (-a, a) \times P$, $a > 0$, where $g_+$, $g_-$ are defined and smooth on $t \geq 0$, $t \leq 0$, and $g_-(\cdot,0)=g_+(\cdot,0)$;
		\item[(ii)] the mean curvatures $H_+$, $H_-$ of the unit normal $\displaystyle{\tfrac{\p}{\p t}}$ at $P$ with respect to $g_+$, $g_-$, satisfy $H_+ \leq H_-$.
	\end{enumerate}
\item[({\bf b})] $Q$ is a disjoint union of compact sets $Q_1,\dots,Q_N$  such that, for each $k = 1, \ldots, N$
	\begin{enumerate}
		\item[(i)] $g$ is in $W^{1,p_k}$ in a neighborhood of $Q_k$  (with respect to a smooth background metric);
		\item[(ii)] $Q_k$ has codimension at least $\ell_k$, with $\ell_k > (\tfrac{2}{n} - \tfrac{1}{p_k})^{-1} >0$, in the sense that $\limsup_{\epsilon \to 0} \epsilon^{-\ell_k} | Q_k(\epsilon)|_g < \infty$; $|Q_k(\epsilon)|_{g}$ denotes its volume.
	\end{enumerate}

\end{itemize}
Here and below $A(\e)$ is the set $ \{ \dist_g(A, \cdot) < \epsilon \}$. If $\e$ is small enough, $P(\e)=(-\e,\e)\times P$.

We want to prove the following:
\begin{thm} \label{t-fsing}
	Let $\Sigma$ be a closed, orientable  manifold of dimension $n-1$, $n\ge 3$, endowed with a Riemannian metric $\gamma$. Suppose  $(\Omega^n, g)$ is a singular fill-in of $(\Sigma,\gamma)$  satisfying  condition $(\mathbf{C})$ with $P, Q$ mentioned in that condition. Then for any $\delta>0$,  there is a smooth metric $h$ on $\Omega$ such that $(\Omega,h)$ is a fill-in of $(\Sigma,\gamma)$ and
 $$
 \int_\Sigma H_h\, d\sigma \ge \int_\Sigma H_g\, d\sigma-\delta.
 $$
 Moreover, if $g$ is not Ricci flat at some point in $\Omega\setminus \sing g$,  or $H_->H_+$ somewhere on $P$, then there is a smooth metric $h$ on $\Omega$ such that $(\Omega,h)$ is a fill-in of $(\Sigma,\gamma)$ with
 $$
 \int_\Sigma H_h\, d\sigma > \int_\Sigma H_g\, d\sigma.
 $$
 In particular, if $n=3$, then
 \be\label{e-fsing-inequality}
\frac1{8\pi} \int_\Sigma H_g d\sigma\le \Lambda(\Sigma,\gamma)
\ee
with equality only if $g$ is flat outside $\sing g$ and $H_+=H_-$ along $P$.

\end{thm}

 \begin{rem} \label{r-p-1} \label{r-sing-b}
	\begin{itemize}
\item[(i)] We are primarily interested in the 3-dimensional case, $n=3$, in which case, in condition $(\mathbf{C})( \mathbf{b})$ above, $\ell = 2 \implies p > 6$ and $\ell = 3 \implies p > 3$.  Also, $\ell \leq n$ is automatic. Moreover, we have $\ell>\tfrac n2$.  It is unclear if one can relax the condition to allow $\ell \searrow 1$. However see Proposition \ref{p-isolated-singularities-inequality-rigidity} below.
\item[(ii)] The condition on $P$ is necessary {for Theorem \ref{t-fsing} to hold true}. One can construct examples which do not satisfy the condition on $P$ and  that \eqref{e-fsing-inequality} fails. See Section \ref{s-schwarzschild}.
\item[(iii)] {The inequality in Theorem \ref{t-fsing} follows from a smoothing procedure that combines ideas from \cite{Lee13, Miao02} that are now well-known among the experts. The rigidity statement requires studying the effects of the boundary mean curvature and nontrivial interior Ricci curvature (Lemma \ref{l-du-boundary}), which is a novel part of our work.}
\end{itemize}	
\end{rem}

In some cases, the fill-in metric $g$ is actually smooth when equality is attained in  \eqref{e-fsing-inequality}. More precisely,

\begin{prop}\label{p-fsing-rigidity}
	Let $(\Omega^3, g)$ be as in Theorem \ref{t-fsing} with $n=3$, and so that it attains equality in \eqref{e-fsing-inequality}. Then $g$ extends smoothly (after perhaps a change of smooth structure) across:
	\begin{enumerate}
		\item $P$, if $P$ is homeomorphic to union of spheres,
		\item $Q$, if $Q$ consists of isolated point singularities,
		\item $Q$, if $g$ is Lipschitz.
	\end{enumerate}
	If $\sing g$ consists of only such singularities, then $\Sigma$ is connected, $\sing g = \emptyset$, and $(\Omega, g)$ is a mean-convex handlebody with flat interior.
\end{prop}

\begin{proof}
	If equality is attained in \eqref{e-fsing-inequality}, then $H_+ = H_-$ on $P$ by Theorem \ref{t-fsing}, and $g$ is flat outside $P$. By \cite[p. 1180--1181]{Miao02}, $g$ extends smoothly across $P$   that are spheres with $H_+ = H_-$. Likewise, $g$ is smooth across $Q$ which consists of  isolated points by \cite{SmithYang92}, and also smooth across $Q$ if $g$ is   by \cite[Theorem 6.1]{ShiTam17}.
	
	If all the singularities in $\sing g$ are as above, then $(\Omega, g) \in \FF(\Sigma, \gamma)$ and the result follows verbatim from \cite[Theorem 1.4]{M-M}.
\end{proof}

To prove   Theorem \ref{t-fsing}, we need to approximate $g$ by smooth metrics with suitable properties.  Let $\sing g=P\cup Q$ as in the theorem, assuming that $Q=Q_1$ with $ p:=p_1$, $\ell:=\ell_1$.  The general case can be proved similarly. Let $\mathfrak{b}$ be a smooth background metric on $\Omega$. Cover $Q$ by finitely many coordinate neighborhood $U_1,\dots, U_N$ with $\ol U_k$ is in the interior of $\Omega$ and  $P\cap (   \cup_{k=1}^N\ol U_k)=\emptyset$. Let $U_{N+1}$ be an open set in $\Omega$ so that $Q\cap \ol U_{N+1}=\emptyset$ and $\Omega=\cup_{k=1}^{N+1}U_k$. Let $\psi_k$ be a partition of unity with $ \mathrm{supp}(\psi_k)\subset U_k$. By \cite[Lemma 3.1]{Lee13}, for each $1\le k\le N$, there is a smooth function $\e\ge\rho_k\ge0$ in $U_k$ such that for small $\e>0$:
\be\label{e-cutoff-1}
  \begin{array}{c}
    \rho_k= \e \text{ in } Q(\e/2)\cap U_k, \text{ and } \rho_k=   0 \text{ in } U_k\setminus Q(\e); \\
    |\p\rho_k| + \e |\p^2\rho_k|\le   C;  \\
  \end{array}
\ee
for some $C$ independent of $\e$ and $k$. Here $\p\rho_k$ and $\p^2\rho_k$ are   derivatives with respect to the Euclidean metric.

\begin{lemma}\label{l-smoothing}
There is $\e_0>0$ such that if $\e_0\ge\e>0$, there is a smooth metric $g_\e$ on $\Omega$ such that
\begin{enumerate}
  \item [(i)] $g_\e=g$ outside $P(\e)\cup Q(\e)$, and $c^{-1} g_\eps \leq  g \leq cg_\e$ for some constant $c>0$ independent of $\e$;

  \item [(ii)] the $W^{1,p}$ norm with respect to $\mathfrak{b}$ of  $g_\e$  in  $Q(\e)$ is less than $c$, for some constant $c>0$ independent of $\e$;
  \item [(iii)]  $|\nabla_{\mathfrak{b}}g_\e|\le c$ in $P(\e)$, and for $z\in P$,
\be\label{e-scalar-1}
	\begin{array}{ll}
    	|R(g_\e)|(z,t) \le c & \text{ for } \frac{\e^2}{100} <|t|\le \e; \\
		-c +\e^{-2}(H_-(z)-H_+(z)) \\
		\qquad \qquad \leq R(g_\e)(z,t) \leq c\e^{-2} & \text{ for } -\frac{\e^2}{400}  <t\le \frac{\e^2}{400}; \\
		-c \leq R(g_\e)(z,t) \leq c\eps^{-2} & \text{ for } \frac{\e^2}{400} < |t| \le \frac{\e^2}{100};
  \end{array}
\ee
for some constant $c>0$ independent of $\e$.
\end{enumerate}
\end{lemma}
\begin{proof} In the following $c>0$ always denote  a constant independent of $\e$. Its meaning may vary from line to line. First assume $\e_0>0$ be small enough so that $P(2\e_0)\cap Q(2\e_0)=\emptyset$ and $P(2\e_0)\cup Q(2\e_0)$ is disjoint from the boundary $\Sigma$.

By \cite[Prop. 3.1]{Miao02}, by choosing a possible smaller $\e_0>0$, then for any $0<\e\le \e_0$,   one can find a metric $\wt g_{\e}$    such that $\wt g_{\e}=g$ outside $P(\e)$, $\wt g_\e$ is smooth in $P(\e_0)$,  $|^\mathfrak{b}\nabla \wt g_{\e}|_\mathfrak{b}\le c$ in $P(\e_0)$ such that the scalar curvature $R(\wt g_\e)$ in $P(\e)$ satisfies \eqref{e-scalar-1} with $\widetilde{g}_\eps$ in place of $g_\eps$ for all $z\in \Sigma$.

Next we want to modify  $\wt g_\e$  near $Q$. Let $\psi_k$ be the partition of unity mentioned before the lemma. Let $g^{k}:=\psi_k\wt g$ and for $1\le k\le N$,
\begin{equation} \label{e-sing-b-approx-0}
	(g^{k}_{\e})_{ij}(x) := \int_{\R^n}\wt g_{ij}^k(x-\lambda\rho_k(x)y) \varphi_{\R^n}(y) \, dy.
\end{equation}
Here
\begin{equation} \label{e-sing-n-dim-mollifier}
	\varphi : \R^n \to [0,\infty) \text{ is smooth}, \spt \varphi \subset B_{1/2}^n(0), \; \int_{\R^n} \varphi = 1,
\end{equation}
and $\lambda>0$ is a constant independent of $\eps$ and $k$, which is small enough for $x \mapsto x - \lambda \rho_k(x) y$ from \eqref{e-sing-b-approx-0} to be a diffeomorphism.

Finally, define
\begin{equation} \label{e-sing-approx-1}
	g_{\eps} := \sum_{k=1}^N \wt g^{k}_{\eps} + \psi_{N+1}\wt g.
\end{equation}

We have, e.g., from \cite[Lemma 4.1]{ShiTam17}, that  for a possibly smaller $\eps_0>0$, then for $\eps \in (0, \eps_0]$,
\begin{equation} \label{e-sing-away}
   \begin{array}{c}
    g_\eps = \wt g  \text{ except on } Q(\eps),\\
    c^{-1} g_\eps \leq  \wt g \leq c g_\eps \text{ on } \Omega\\
    \text{$W^{1,p}$ norm of  $g_\e$  in  $Q(\e)$ is less than $c$.}
   \end{array}
\end{equation}
for a constant $c > 0$ that is independent of $\eps > 0$. Combining this with the properties of $\wt g$, we see that $g_\e$ satisfies the required properties of the lemma.
\end{proof}

In order to understand the effect of the Ricci tensor and $H_--H_+$, we also introduce the following. Fix $x_0 \in \Omega \setminus (\Sigma \cup \sing g)$ and $r > 0$ small enough that $5r < \operatorname{inj}_g (x_0)$ (the injectivity radius of $g$ at $x_0$) and $B^g_{5r}(x_0) \cap (\Sigma \cup \sing g) = \emptyset$.

Choose  $\eps_0 > 0$ be sufficiently small in the above so that
\begin{equation} \label{e-sing-avoidance}
	(P(\eps_0) \cup Q(\eps_0)) \cap B^g_{4r}(x_0) = \emptyset.
\end{equation}
Let $\chi : [0,\infty) \to [0,\infty)$ be a smooth nonincreasing function with
\begin{equation} \label{e-chi}
	\chi = 1 \text{ on } [0,1], \; \chi = 0 \text{ on } [2,\infty), \; |\chi'|^2 \leq C \chi,
\end{equation}
with $C$ some absolute constant. Set:
\begin{equation} \label{e-perturbation-tensor}
	h := \chi(\dist_g(\cdot, x_0)/r) \mathring \Ric(g_\eps) = \chi(\dist_g(\cdot, x_0)/r) \mathring \Ric(g);
\end{equation}
because $g_\e=g$ outside $P(\e)\cup Q(\e)$. Here $\mathring{\Ric}(g)$ is the traceless part of the Ricci tensor of $g$, etc.

For $\tau > 0$, define
\[ g_{\eps,\tau} := g_\eps - \tau h. \]
If $\tau$ is small enough, independent of $\eps$, then  $g_{\eps,\tau}$ also satisfies Lemma \ref{l-smoothing}  with $g_\eps$ replaced by $g_{\eps,\tau}$ with constants $c$ being independent of $\e, \tau$ provided they are small enough.  Observe that the Sobolev constants in the Sobolev inequality for $g_\e$ and $g_{\e,\tau}$  are uniformly bounded.

Let $\rho_k$ be as in  \eqref{e-cutoff-1} and $\psi_k$ be the partition of unity on $U_1,\dots, U_{N+1}$ mentioned  above.  Let
\[ \eta := \eps^{-1} \sum_{k=1}^N \rho_k \psi_k. \]
Then $\eta=1$ on $Q(\eps/2)$ and $\eta=0$ outside $Q(\eps)$.
\begin{lemma} \label{l-scalar-est}
	The following statements hold.
	\begin{enumerate}
		\item[(i)] There is $\a>0$ such that if   $\tau_0, \eps_0>0$ are small enough so that if $\eps \in (0,\eps_0]$, $\tau \in (0,\tau_0]$, $\e^\a\le \tau^2$, then for any Lipschitz function $u$ on $\Omega$:
			\begin{align}\label{e-scalar-est-1}
				& \int_\Omega R(g_{\eps,\tau}) u^2 \, \dvol_{g_{\eps,\tau}}\\
				& \geq \int_{\Omega \setminus(P(\e)\cup   Q(\eps))}   R(g) u^2 \, \dvol_{g_{\eps,\tau}}    + c^{-1}\e^{-2} \int_{P (\frac1{400}\eps^2)} (H_- - H_+) u^2 \, \dvol_{g_{\eps,\tau}} \nonumber \\
				&  + \tfrac12 \tau \int_{B^g_{4r}(x_0)} \chi |\mathring{\Ric}(g)|^2 u^2 \, \dvol_{g_{\eps,\tau}}       - c \tau^\frac32  \left[ \int_{Q(\eps) \cup B^g_{4r}(x_0)} u^{\tfrac{2n}{n-2}} \, \dvol_{g_{\eps,\tau}} \right]^{\tfrac{n-2}{n}} \nonumber  \\
				&   - c\tau^\frac12 \int_\Omega |\nabla_{g_{\eps,\tau}} u|^2 \, \dvol_{g_{\eps,\tau}}, \nonumber
			\end{align}
		for $c>0$ which is  independent of $\eps, \tau$.
		\item[(ii)] There is $\beta>0$ such that if   $\tau_0, \eps_0>0$ are small enough so that if $\eps \in (0,\eps_0]$, $\tau \in (0,\tau_0]$, $\e^\beta\le \tau^2$, then for  any Lipschitz function $u \geq 0$ on $\Omega$ with $u|_{\p \Omega} = 0$,
		\begin{equation} \label{e-scalar-est-2}
		\int_\Omega R(g_{\eps,\tau}) u \, \dvol_{g_{\eps,\tau}} \geq -c \tau  \left[ \int_\Omega |\nabla_{g_{\eps,\tau}} u|^2 \, \dvol_{g_{\eps,\tau}} \right]^{\tfrac12},
		\end{equation}
		for $c>0$ which is independent of $\eps, \tau$.
	\end{enumerate}
\end{lemma}

\begin{proof}
First we prove \eqref{e-scalar-est-1}. Since $\eta=0$ outside $Q(\e)$ and $g_\e=g$ outside $P(\e)\cup Q(\e)$.  We have
\begin{align} \label{e-scalar-est-1-integral}
	& \int_\Omega R(g_{\e,\tau}) u^2 \, \dvol_{g_{\e,\tau}} \\
	& = \int_{Q(\e)} \eta R(g_{\eps,\tau}) u^2 \, \dvol_{g_{\eps,\tau}} + \int_{\Omega\setminus Q(\e)} (1-\eta) R(g_{\eps,\tau})u^2 \, \dvol_{g_{\eps,\tau}} \nonumber \\
	& = \underbrace{\int_{P(\eps)}   R(g_{\eps}) u^2 \, \dvol_{g_{\eps}}}_{=: \, \mathbf{I}_A} + \underbrace{\int_{Q(\eps)} \eta R(g_{\eps}) u^2 \, \dvol_{g_{\eps}}}_{=: \, \mathbf{I}_B} + \underbrace{\int_{B^g_{4r}(x_0)}( R(g_{\eps,\tau})-R(g) )u^2 \, \dvol_{g_{\eps,\tau}}}_{=: \, \mathbf{II}} \nonumber \\
	& \qquad +  \int_{\Omega \setminus (P(\e)\cup Q(\eps)) }   R(g) u^2 \, \dvol_{g}. \nonumber
\end{align}
Here we have used the fact that $ P(\e)\cup Q(\e)$ is disjoint from $B^g_{4r}(x_0)$.

First, we have
\begin{align} \label{e-scalar-est-1-IA}
	\mathbf{I}_A
		& = \int_{P(\eps)}   R(g_{\eps}) u^2 \, \dvol_{g_{\eps}} \\
		& \geq -c \int_{P(\eps)} u^2 \, \dvol_{g_{\eps}}    + \eps^{-2}\int_{P(\eps^2/400)} (H_- - H_+)    u^2 \, \dvol_{g_{\eps}} \nonumber \\
		& \geq -c \eps^{\tfrac{2}{n}} \left[ \int_{P(\eps)} u^{\tfrac{2n}{n-2}} \, \dvol_{g_{\eps}} \right]^{\tfrac{n-2}{n}} + \e^{-2}\int_{P(\frac1{400}\eps^2)} (H_- - H_+) u^2 \, \dvol_{g_{\eps}} \nonumber,
\end{align}
where we have used Lemma \ref{l-smoothing}(iii).

Next, we have
\begin{equation} \label{e-scalar-est-1-IB-split}
	\mathbf{I}_B = \int_{Q(\eps)} \eta R(g_{\eps}) \, \dvol_{g_\eps} = \eps^{-1}\int_{Q(\e)} \sum_{k=1}^N \rho_k \psi_k R(g_\eps) u^2 \, \dvol_{g_\eps}.
\end{equation}
We estimate each term separately. Recall that, in each coordinate chart $U_k$, $k = 1, \ldots, N$,
\[ R(g_\eps) = (g_\eps)^{-1} * \partial \Gamma_\eps + (g_\eps)^{-1} * \Gamma_\eps * \Gamma_\eps, \]
where the partial derivatives are with respect to Euclidean coordinates, $\Gamma_\eps$ denote the Christoffel symbols of $g_\eps$ in these coordinates, and the ``$*$'' notation above stands for linear combinations of the form $g^{ij}_\eps \partial_k[(\Gamma_\eps)^s_{pq}]$, $g^{ij}_\eps (\Gamma_\eps)^s_{pq} (\Gamma_\eps)^w_{uv}$, respectively. Using this expression, we can estimate each term of \eqref{e-scalar-est-1-IB-split} by integrating by parts.

For $\kappa > 0$, we have:
\begin{align} \label{e-scalar-est-1-IB-term}
	  &\eps^{-1} \int_{Q(\eps)} \rho_k \psi_k R(g_\eps) u^2 \, \dvol_{g_\eps} \\
	& = \eps^{-1} \int_{U_k \cap Q(\eps)} \rho_k \psi_k \big[ (g_\eps)^{-1} * \partial \Gamma_\eps + (g_\eps)^{-1} * \Gamma_\eps * \Gamma_\eps \big] u^2 \, \dvol_{g_\eps} \nonumber \\
	& \geq -c \int_{U_k \cap Q(\eps)} \big[ (\eps^{-1} |\partial g_\eps| + |\partial g_\eps|^2) u^2 + u |\nabla_{g_\eps} u| |\partial g_\eps| \big] \, \dvol_{g_\eps} \nonumber \\
	& \geq - c \left[ \eps^{-1 + (\tfrac{2}{n} - \tfrac{1}{p}) \ell} + (1+\kappa^{-1}) \eps^{(\tfrac{2}{n} - \tfrac{2}{p})\ell} \right] \left[ \int_{U_k \cap Q(\eps)} u^{\tfrac{2n}{n-2}} \, \dvol_{g_\eps} \right]^{\tfrac{n-2}{n}} \nonumber \\
	& \qquad - c \kappa \int_{U_k \cap Q(\eps)} |\nabla_{g_\eps} u|^2 \, \dvol_{g_\eps} \nonumber,
\end{align}
where we have used  the generalized H\"older inequality,    \eqref{e-sing-away}, $|\Gamma_\eps| = O(|\partial g_\eps|)$,   and the fact that $Q$ has co-dimension at least $\ell$  and Lemma \ref{l-smoothing}(ii). Here and below, $c > 0$ is a constant independent of $\eps$, $\tau$, $\kappa$. The meaning of $c$ may change from line to line. Summing \eqref{e-scalar-est-1-IB-term} over $k= 1, \ldots, N$, we get
\begin{align}\label{e-scalar-est-1-IB}
	\mathbf{I}_B
 \geq - c \tau^\frac12 \int_{  Q(\eps)} |\nabla_{g_\eps} u|^2 \, \dvol_{g_\eps}    - c  \tau^\frac32 \left[ \int_{Q(\eps)} u^{\tfrac{2n}{n-2}} \, \dvol_{g_\eps} \right]^{\tfrac{n-2}{n}}
\end{align}
if   $\e^\a\le \tau^2$, $\kappa=\tau^\frac12$, where $\a$ is   of $(\tfrac{2}{n} - \tfrac{2}{p})\ell$ and $ -1 + (\tfrac{2}{n} - \tfrac{1}{p })$ which is positive.

We now proceed to estimate $\mathbf{II}$ from \eqref{e-scalar-est-1-integral}. To do so, let's recall:

\begin{lemma*}[{\cite[Proposition 4]{BrendleMarques11}}]
	Let $(\Omega^n,  g)$ be a smooth Riemannian manifold. If $\bar g=g+h$ with $ |h|_{ g}\le \frac12$, then the scalar curvatures satisfy:
	\bee
	R(\bar g) - R(g) = - \divg_g(\divg_g(h)) +  \Delta_g \tr_g h - \la h, \Ric(g) \ra_g + F
	\eee
	where $|F|\le C\lf( |\nabla h|^2+|h|_g|\nabla^2h|_g+|\Ric(g)||h|^2_g\ri)$ and $C$ depends only on $n$. Here $\nabla$ is the covariant derivative with respect to $g$.
\end{lemma*}

Hence, for small $\tau$, we have
\[ R(g_{\eps,\tau}) = R(g) + \tau   \div_g(\div_g(h)) + \tau \la h,\Ric(g)\ra_g + E_1, \]
on $B^g_{4r}(x_0)$ where $|E_1|\le c\tau^2$, with $c$ independent of $\eps$, $\tau$. Moreover,
\[ \dvol_{g_{\eps,\tau}} = (1 + E_2) \, \dvol_g \]
on $B^g_{4r}(x_0)$, where $|E_2| \leq c \tau^2$, with $c$ independent of $\eps$, $\tau$, because the perturbation tensor $h$ from \eqref{e-perturbation-tensor} is traceless. Moreover, since
\[ \mathring{\Ric}(g) = \Ric(g) - \tfrac{1}{n} R(g) = (\Ric(g) - \tfrac{1}{2} R(g) g) + \tfrac{n-2}{2n} R(g) g, \]
and the first term of the right hand side is divergence-free, we have
\[ \divg_g \divg_g h = \divg_g \mathring{\Ric}(g)(\nabla_g \chi, \cdot) + \tfrac{n-2}{2n} \divg_g (\chi \nabla_g R(g)). \]
Using the fact that $h$, its derivatives, and $\Ric(g)$ are bounded by a constant independent of $\eps$, $\tau$ in $B^g_{4r}(x_0)$, for any $\kappa \in (0, 1)$ we have
\begin{align} \label{e-scalar-est-1-II}
	\mathbf{II}
		& = \int_{B_{4r}^g(x_0)} (R(g_{\eps,\tau})-R(g)) u^2 \, \dvol_{g_{\eps,\tau}} \\
		& = \int_{B_{4r}^g(x_0)} (R(g) + \tau \divg_g \divg_g h + \tau \la h, \Ric(g) \ra_g + E_1) (1 + E_2) u^2 \, \dvol_g \nonumber \\
		& \geq \tau \int_{B^g_{4r}(x_0)} (\divg_g \div_g h + \la h, \Ric(g) \ra_g) u^2 \, \dvol_g \nonumber - c \tau^2 \int_{B^g_{4r}(x_0)} u^2 \, \dvol_g \nonumber \\
		& \geq \tau \int_{B^g_{4r}(x_0)} \chi |\mathring{\Ric}(g)|^2 u^2 \, \dvol_g - c \tau^2 \int_{B^g_{4r}(x_0)} u^2 \, \dvol_g\nonumber \\
		& \qquad - c \tau \int_{B^g_{4r}(x_0)} u |\nabla_g u| (|\nabla_g \chi| |\mathring{\Ric}(g)| + \chi |\nabla_g R(g)|) \, \dvol_g \nonumber \\
		& \geq (1 - c\kappa) \tau \int_{B^g_{4r}(x_0)} \chi |\mathring{\Ric}(g)|^2 u^2 \, \dvol_g - \kappa^{-1} \tau \int_{B^g_{4r}(x_0)} |\nabla_{g_{\eps,\tau}} u|^2 \, \dvol_{g_{\eps,\tau}} \nonumber \\
		& \qquad - c(\tau^2 + \kappa \tau) \int_{B^g_{4r}(x_0)} u^2 \, \dvol_{g_{\eps,\tau}}, \nonumber\\
& \geq \frac\tau 2  \int_{B^g_{4r}(x_0)} \chi |\mathring{\Ric}(g)|^2 u^2 \, \dvol_g - \tau^\frac12 \int_{B^g_{4r}(x_0)} |\nabla_{g_{\eps,\tau}} u|^2 \, \dvol_{g_{\eps,\tau}} - c\tau^\frac32 \int_{B^g_{4r}(x_0)} u^2 \, \dvol_{g_{\eps,\tau}}, \nonumber
\end{align}
if we take $\kappa=\tau^\frac12$,   where $c$ continues to denote a constant independent of $\eps$, $\tau$, provided $\tau_0$ is small enough.

Altogether, \eqref{e-scalar-est-1-integral}, \eqref{e-scalar-est-1-IA}, \eqref{e-scalar-est-1-IB}, \eqref{e-scalar-est-1-II} imply \eqref{e-scalar-est-1}.

Now let's prove \eqref{e-scalar-est-2}. As before,  $ R(g)\ge0$ outside $\sing g$, $u\ge0$, so
\begin{align} \label{e-scalar-est-2-integral}
	  \int_\Omega R(g_{\eps,\tau}) u \, \dvol_{g_{\eps,\tau}}
	& \geq \underbrace{\int_{P(\eps)}   R(g_\eps) u \, \dvol_{g_\eps}}_{=: \, \mathbf{IV}_A} + \underbrace{\int_{Q(\eps)} \eta R(g_\eps) u \, \dvol_{g_\eps}}_{=: \, \mathbf{IV}_B} \nonumber \\
	& \qquad   + \underbrace{\int_{B^g_{4r}(x_0)}( R(g_{\eps,\tau})-R(g)) u \, \dvol_{g_{\eps,\tau}}.}_{=: \, \mathbf{V}}
\end{align}
Arguing as for $\mathbf{I}_A$ in \eqref{e-scalar-est-1-IA}, we can estimate
\begin{equation} \label{e-scalar-est-2-IVA}
	\mathbf{IV}_A \geq -c \eps^{\tfrac{n+2}{2n}} \left[ \int_{P(\eps)} u^{\tfrac{2n}{n-2}} \, \dvol_{g_\eps} \right]^{\tfrac{n-2}{2n}} \geq -c \eps^{\tfrac{n+2}{2n}} \left[ \int_{\Omega} |\nabla_{g_\eps} u|^2 \, \dvol_{g_\eps} \right]^{\tfrac12};
\end{equation}
we used the Sobolev inequality for functions with value zero boundary value on $\Omega$, together with Lemma \ref{l-smoothing}(i).

Arguing as for $\mathbf{I}_B$ in \eqref{e-scalar-est-1-IB-split}, \eqref{e-scalar-est-1-IB-term}, we have, for each $k = 1, \ldots, N$:
\begin{align} \label{e-scalar-est-2-IVB-prelim}
	& \eps^{-1} \int_{Q(\eps)} \rho_k \psi_k R(g_\eps) u \, \dvol_{g_\eps} \\
	& \geq -c \int_{U_k \cap Q(\eps)} \big[ (\eps^{-1} |\partial g_{\eps}| + |\partial g_{\eps}|^2) u + |\nabla_{g_\eps} u| |\partial g_\eps| \big] \, \dvol_{g_\eps} \nonumber \\
	& \geq -c \left[ \eps^{-1 + (\tfrac{n+2}{2n} - \tfrac{1}{p}) \ell} + \eps^{(\tfrac{n+2}{2n} - \tfrac{2}{p}) \ell} \right] \left[ \int_{U_k \cap Q(\eps)} u^{\tfrac{2n}{n-2}} \, \dvol_{g_\eps} \right]^{\tfrac{n-2}{2n}} \nonumber \\
	& \qquad \qquad - c \eps^{(\tfrac12 - \tfrac{1}{p}) \ell} \left[ \int_{U_k \cap Q(\eps)} |\nabla_{g_\eps} u|^2 \, \dvol_{g_\eps} \right]^{\tfrac12} \nonumber\\
&\geq -c\tau^2\lf(\left[ \int_{U_k \cap Q(\eps)} u^{\tfrac{2n}{n-2}} \, \dvol_{g_\eps} \right]^{\tfrac{n-2}{2n}}+\lf[\int_{U_k \cap Q(\eps)} |\nabla_{g_\eps} u|^2 \, \dvol_{g_\eps}\ri]^\frac12\ri) \nonumber
\end{align}
if $\e^\beta\le \tau^2$, where $\beta$ is the minimum of $-1 + (\tfrac{n+2}{2n} - \tfrac{1}{p}) \ell$, $(\tfrac12 - \tfrac{1}{p}) \ell$ and $(\tfrac{n+2}{2n} - \tfrac{2}{p}) \ell$ which is positive.
Thus, relaxing $U_k \cap Q(\eps)$ to $\Omega$ and using the Sobolev inequality for  functions with value zero boundary value on $\Omega$ together with Lemma \ref{l-smoothing}(i), we find that
\begin{equation} \label{e-scalar-est-2-IVB}
	\mathbf{IV}_B \geq -c\tau^\frac12  \left[ \int_\Omega |\nabla_{g_\eps} u|^2 \, \dvol_{g_\eps} \right]^{\tfrac12},
\end{equation}
with $c$ independent of $\eps$, $\tau$.

Finally, arguing as we did for $\mathbf{IV}$ in \eqref{e-scalar-est-1-II}, we get:
\begin{equation} \label{e-scalar-est-2-V}
	\mathbf{V} \geq - c \tau \left[ \int_\Omega |\nabla_{g_{\eps,\tau}} u|^2 \, \dvol_{g_{\eps,\tau}} \right]^{\tfrac12}.
\end{equation}
Thus, \eqref{e-scalar-est-2} follows from \eqref{e-scalar-est-2-integral}, \eqref{e-scalar-est-2-IVA}, \eqref{e-scalar-est-2-IVB},  \eqref{e-scalar-est-2-V}, and Lemma \ref{l-smoothing}(i). This completes the proof of Lemma \ref{l-scalar-est}.
\end{proof}

\begin{proof}[Proof of Theorem \ref{t-fsing}] Fix $x_0\in \Omega\setminus\sing g$ and choose $r>0$ and construct $g_{\e,\tau}$ as above.
Let $\tau_i \in (0,1)$, $\tau_i \to 0$. Define $\eps_i > 0$ so that
\begin{equation} \label{e-epsilon-i}
	\eps_i^\alpha, \e_i^\beta \leq \tau_i^2
\end{equation}
where $\a,\beta$ are as in  Lemma \ref{l-scalar-est}. For notational simplicity, set $g_i := g_{\eps_i,\tau_i}$. Let $c_n=4(n-1)/(n-2)$. By  Lemma \ref{l-scalar-est}(i), using Sobolev inequality, we conclude that if $i$ is large enough, then
\be\label{e-eigenvalue}
		\int_\Omega |\nabla_{g_i} u|^2 + c_n R(g_i) u^2 \, \dvol_{g_i} \geq \lambda \int_\Omega |\nabla_{g_i}u|^2 \dvol_{g_i}
\ee
for some $\lambda>0$ for all $i$ for any smooth function with compact support in $\Omega$. By standard arguments, there is $i_0>0$ such that for all $i = i_0, i_0 + 1, \ldots$ there is a smooth $u_i>0$ solving
	\begin{equation} \label{e-conformally-scalar-flat}
		\begin{array}{rl}
			- \Delta_{g_i} u_i + \tfrac{n-2}{4(n-1)} R(g_i) u_i = 0 & \text{in } \Omega, \\
			u_i = 1 & \text{on } \Sigma.
		\end{array}
	\end{equation}
	Hence $u_i^{4/(n-2)} g_i$ is a smooth metric on $\Omega$, has scalar curvature everywhere equal to zero, induces the same metric as $g_i$ on $\Sigma = \p \Omega$, and the mean curvature of $\Sigma$ with respect to $g_i$ satisfies
	\begin{equation} \label{e-conformally-scalar-flat-bdry}
		H_{u_i^{4/(n-2)} g_i} = H_g + 2(n-1) \displaystyle{\tfrac{\p u_i}{\p \nu}},
	\end{equation}
	where $\nu$ is the unit outward normal of $g$. The theorem follows from the following lemma because $x_0, r$ are arbitrary.
\end{proof}

\begin{lemma}\label{l-du-boundary} Let $u_i$ be as in the  proof of Theorem \ref{t-fsing}, we have:
\begin{itemize}
  \item [(i)] $
		\liminf_{i\to\infty} \left( \inf_{\Sigma} \tfrac{\p}{\p \nu} u_i \right) \geq 0$, where $\nu$  is the unit outward normal of $\Sigma$ with respect to $g$.
  \item [(ii)] If $g$ is not Ricci flat at $x_0$, or if $H_->H_+$ somewhere in $P$, then $\int_\Sigma  \displaystyle{\tfrac{\p u_i}{\p \nu}}d\sigma>0$ if $i$ is large enough.

\end{itemize}
\end{lemma}
\begin{proof} (i) Let $f_+$ denote  the positive part of a function $f$, i.e., $f_+ := \max\{f,0\}$. Multiply
	\eqref{e-conformally-scalar-flat} by $v_i=(u_i-1)_+$ and integrate by parts. Using Lemma \ref{l-scalar-est}(ii)  we have that for large $i$:
	\[
	\lambda\int_\Omega |\nabla_{g_i} v_i|^2 \, \dvol_{g_i}\le	\int_\Omega |\nabla_{g_i} v_i|^2 + \tfrac{n-2}{4(n-1)} R(g_i) v_i^2 \, \dvol_{g_i} \leq  c \tau_i \left[ \int_\Omega |\nabla_{g_i} v_i|^2 \, \dvol_{g_i} \right]^{\tfrac12}
	\]
	for some $c > 0$ independent of $i$, where $\lambda>0$ is the constant in \eqref{e-eigenvalue}. Hence
	\[
		\left[ \int_\Omega |\nabla_{g_i} v_i|^2 \, \dvol_{g_i} \right]^{\tfrac12} \leq c \tau_i.
	\]
	By the Poincar\'e--Sobolev inequality, we have:
\begin{equation} \label{e-u-L2-bound}
		 \lim_{i\to\infty} \int_{\Omega} [(u_i - 1)_+]^2 \, \dvol_g = 0.
	 \end{equation}
In particular, the $L^2$ norms $\int_\Omega u_i^2 \, \dvol_{g_i}$ are uniformly bounded.
Recall that, by \eqref{e-sing-away}, $g_i = g$ near $\Sigma$. Consider a collar $V$ near $\Sigma$, without loss of generality $V=\Sigma\times[0,1]$, where the metric $g$ is of the form $dt^2+h_t$; here $t$ denotes the distance from $\Sigma$ and $h_t$ denotes a smooth metric on $\Sigma\times\{t\}$.    Hence, by the mean value inequality and Schauder boundary estimates, $\limsup_{i\to\infty} \Vert u_i \Vert_{C^2(V)} < \infty.$

	If (i) were false, then by Arzel\`a-Ascoli we could to pass to a subsequence that converges in $C^1(V)$ to a function $u_\infty$ on $V$. If $\tfrac{\p}{\p \nu} u_\infty(x_\infty) < 0$ at some $x_\infty \in \Sigma$,  by \eqref{e-u-L2-bound}  $u_\infty \leq 1$ on $V$ since the $L^2$ topology is weaker than $C^1$, so this is impossible. Hence (i) is true.

(ii)	Suppose
$
\int_\Sigma \tfrac{\p}{\p \nu} u_i \, d\sigma \le 0
$
for infinitely many $i$.  We may  suppose this is true for all $i$ large enough. We want to prove that $g$ is Ricci flat at $x_0$ and $H_-=H_+$ on $P$. First note that since $u_i=1$ on $\Sigma$, by the Sobolev inequality, we have
\bee
\lf(\int_\Omega u_i^{\tfrac{2n}{n-2}}\dvol_{g_i} \ri)^{\tfrac{n-2}{n}}\le c\lf(\int_\Omega |\nabla_{g_i}u_i|^2 \dvol_{g_i}+1\ri)
\eee
here we have used the fact that $c g_i\le g\le c^{-1}g_i$ for some $c>0$ independent of $i$.

Multiplying  \eqref{e-conformally-scalar-flat} by $u_i$,  integrating by parts, and using using $u_i \equiv 1$ on $\Sigma$, Lemma \ref{l-scalar-est}(i) we have,
\begin{align} \label{e-fsing-inequality-chain}
	0
		& \geq \int_\Sigma \tfrac{\p}{\p \nu} u_i \, d\sigma \nonumber \\
		& = \int_\Omega |\nabla_{g_i} u_i|^2 + c_n R(g_i) u_i^2 \, \dvol_{g_i} \nonumber \\
		& \geq \int_{\Omega \setminus (P(\eps) \cup Q(\eps))}   R(g) u_i^2 \, \dvol_{g_i}    + c^{-1}\e_i^{-2} \int_{P (\frac1{400}\eps^2)} (H_- - H_+)   u^2_i \, \dvol_{g_i} \nonumber \\
		& \qquad + \tfrac12 \tau_i \int_{B^g_{4r}(x_0)} \chi |\mathring{\Ric}(g)|^2 u^2 \, \dvol_{g_i} +(1 - c\tau_i^\frac12) \int_\Omega |\nabla_{g_i} u_i|^2 \, \dvol_{g_i}\nonumber \\
		& \qquad - c \tau_i^\frac32  \left[ \int_{Q(\eps) \cup B^g_{4r}(x_0)} u_i^{\tfrac{2n}{n-2}} \, \dvol_{g_i} \right]^{\tfrac{n-2}{n}}  \nonumber \\
		& \geq \int_{\Omega \setminus (P(\eps) \cup Q(\eps) )}   R(g) u_i^2 \, \dvol_{g_i}    + c^{-1}\e_i^{-2} \int_{P (\frac1{400}\eps^2)} (H_- - H_+)   u^2_i \, \dvol_{g_i} \nonumber \\
		& \qquad + \tfrac12 \tau_i  \int_{B^g_{4r}(x_0)} \chi |\mathring{\Ric}(g)|^2 u_i^2 \, \dvol_{g_i}     +(1 - c\tau_i^\frac12) \int_\Omega |\nabla_{g_i} u_i|^2 \, \dvol_{g_i} \nonumber \\
		& \qquad - c \tau_i^\frac32
\end{align}
Since $R(g) \geq 0$ outside $\sing g$, $H_+ \leq H_-$ on $P$, we have
	\begin{equation} \label{e-fsing-energy-decay}
		\lim_{i\to\infty} \int_\Omega |\nabla_{g_i}u_i|^2 \, \dvol_{g_i} = 0.
	\end{equation}
	By   the Sobolev inequality on $u_i-1$, and the Harnack inequality, we find that, after passing to a subsequence, $u_i \to 1$ locally uniformly away from $P\cup Q$. Therefore, \eqref{e-fsing-inequality-chain} further implies that $R(g) \equiv 0$ on $\Omega \setminus (P \cup Q)$, $\mathring{\Ric}(g)(x_0)=  0$. Hence $g$ is Ricci flat on $x_0$.

It remains to prove $H_+ = H_-$ along $P$.  Note that on $P(\e_0)=(-\e_0,\e_0)\times P$,  $c(dt^2+h)\le g_i\le c^{-1}(dt^2+h)$ for some $c>0$ independent of $i$,  where $h$ is the induced metric of $g=g_+=g_-$ on $P$. Let $d\sigma_h$ be the volume form of $h$ on $P$. Since $u_i(z,\e_0)\to 1$ uniformly on $\{\e_0\}\times P$,
\begin{align*}
&\e_i^{-2} \int_{P (\frac1{400}\eps^2_i)} (H_- - H_+)   u^2_i \, \dvol_{g_i}\\
\ge &c_1\e_i^{-2}\int_0^{\frac1{400}\eps^2_i}\int_{z\in P}(H_- - H_+)(z)u^2_i(z,t) d\sigma_h(z) dt\\
=&\frac{c_1}{400} \int_P(H_- - H_+)u_i^2(z,\e_0) d\sigma_h\\
&+c_1\e_i^{-2}\int_0^{\frac1{400}\eps^2_i}\int_{z\in P}(H_- - H_+)(z)(u^2_i(z,t)-u_i^2(z,\e_0) )d\sigma_h(z) dt\\
\ge &c_2\int_P(H_- - H_+)d\sigma_h-c_2\e_i^{-2}\int_0^{\frac1{400}\eps^2_i}\int_{z\in P}|(u^2_i(z,t)-u_i^2(z,\e_0))|d\sigma_h(z) dt\\
\ge&c_2\int_P(H_- - H_+)d\sigma_h-2c_2\e_i^{-2}\int_0^{\frac1{400}\eps^2_i}\int_{z\in P}\lf(\int_0^{\e_0}|u_i\p_tu_i |(z,s)ds\ri) d\sigma_h(z) dt\\
\ge &c_2\int_P(H_- - H_+)d\sigma_h-c_3  \int_{z\in P}\lf(\int_0^{\e_0}|u_i\p_t u_i |(z,s)ds\ri) d\sigma_h(z).
\end{align*}
where $c_1, c_2, c_3>0$ are constants independent of $i$.
By \eqref{e-fsing-energy-decay}, Sobolev inequality applied to $u_i-1$ we conclude that $\int_\Omega u_i^2 dv_{g_i}$ is uniformly bounded. By \eqref{e-fsing-energy-decay} and the Schwarz inequality, we conclude that $ \int_{z\in P}\lf(\int_0^{\e_0}|u_i\p_t u_i |(z,s)ds\ri) d\sigma_h(z)\to0$, as $i\to\infty$. Combine this with \eqref{e-fsing-inequality-chain}, we conclude that $H_-=H_+$ along $P$.
\end{proof}

 \section{Isolated singularities in dimension three}\label{s-isolated-singularities}

{In Theorem \ref{t-fsing}, the metric is assumed to satisfy condition   {\bf (C)}. In 
 {particular,}
 in dimension three the metric is assumed to be 
 $W^{1,p}$, $p>3$, near
 {codimension-three}
 singularities. 
 However, 
 {the conclusion}
 	of the theorem is still true in dimension three for \emph{isolated point singularities} with {the} weaker assumption that the metric is 
 	$L^\infty$ as in Definition \ref{d-fsing}. Namely, we have the following:}

\begin{prop} \label{p-isolated-singularities-inequality-rigidity}
	Let $(\Omega^3, g)$ be a singular fill-in of $(\Sigma, \gamma)$ such that $\mathcal{H}^0(\sing g) < \infty$. Then $\FF(\Sigma, \gamma) \neq \emptyset$, and
	\begin{equation} \label{e-isolated-singularities-inequality}
		\frac{1}{8\pi} \int_\Sigma H_g \, d\sigma \leq \LL(\Sigma, \gamma).
	\end{equation}
	If equality holds in \eqref{e-isolated-singularities-inequality}, then $\sing g = \emptyset$, $\Sigma$ is connected, and $(\Omega, g)$ is a mean-convex handlebody with a flat interior.
\end{prop}

\begin{rem} \label{r-isolated-singularities}
{
{The proposition is not true} 
if the metric is not 
$L^\infty$. Let $m>0$,
{and}
consider $M=B(1)\setminus\{0\}$ in $\R^3$, with metric given by
$$
g=d\rho^2+\phi^4 r^2h_0
$$
where $\phi= 1-\frac{m}{2r}$, $h_0$ is the standard metric of the unit sphere, {and} $r$, $\rho$ are related by {$d\rho =\phi^2 dr$} with $\rho=0$ when $r=\frac12m$. Then $(M,g)$ is isometric
{to a portion} 
 $r_0>r>\frac12m$ (for some $r_0$) of the Schwarzschild manifold  with negative mass $-m$.
By the computations in
{Section}
\ref{s-schwarzschild-computations}, we see that the conclusion in the proposition is false for $(M,g)$
{if viewed as a potential ``singular fill-in'' of a round sphere.}
 On the other hand, near the singular point $\rho=0$,
$$
g=d\rho^2+c\rho^\frac43(1+O(\rho^\frac23))h_0.
$$
Hence the metric is not bounded below (thus, not $L^\infty$) near the singular point, {so it is not an admissible singular fill-in per Definition \ref{d-fsing}.}}
\end{rem}

The proof we give relies on the idea behind the class $\F$ of fill-ins; see Definition \ref{d-f-l}. The unprescribed minimal surfaces at the boundary will, essentially, be used to ``hide'' the singular points.

\begin{proof}[Proof of Proposition \ref{p-isolated-singularities-inequality-rigidity}]
	The result is trivial when $\sing g = \emptyset$, so let's assume that $\sing g \neq \emptyset$. We will explicitly construct an element of $\Fcirc(\Sigma, \gamma)$; thus, $\FF(\Sigma, \gamma)$ is also nonempty by Proposition \ref{p-equality-l-ll}.
	
	Let $\psi$ be the Greens function for the Laplacian on $(\Omega^3, g)$ with Dirichlet boundary data on $\partial \Omega$, and $\psi(x) \to \infty$ as $\dist_g(x, \sing g) \to 0$. We note that such Greens functions exist for $g$ with isolated 
	{$L^\infty$}
	point singularities; see, e.g., \cite{LSW63}. 
	Moreover, $\psi \in C^\infty_{\loc}(\Omega \setminus \sing g)$ and
	\begin{equation} \label{e-isolated-point-singularities-greens}
		c^{-1} \dist_g(\cdot, \sing g)^{-1} \leq \psi \leq c \dist_g(\cdot, \sing g)^{-1}, \; x \in \Omega,
	\end{equation}
	where $c$ depends on the original metric $g$.
	
	Define $\wh g_{\epsilon} = (1 + \epsilon \psi)^4 g$. Note that
	\begin{equation} \label{e-isolated-point-singularities-mean-curv-limit}
		\lim_{\epsilon \to 0} H_{\wh g_\epsilon} = H_g \text{ uniformly on } \Sigma
	\end{equation}
	and, since $R(g) \geq 0$ on $\Omega \setminus \sing g$,
	\begin{equation} \label{e-isolated-point-singularities-scalar-curv}
		R(\wh g_\epsilon) \geq 0 \text{ on } \Omega \setminus \sing g.
	\end{equation}
	
	Fix a family of disjoint open neighborhoods of the points in $\sing g$ (one for each point) labeled $\{U_p\}_{p \in \sing g}$, so that each $U_p \subset M \setminus \Sigma$ is diffeomorphic to a 3-ball.
	
	\begin{clai*}
		Let $p \in \sing g$. There exists $\Phi_p : \R^3 \setminus \overline{B}_1(0) \to U_p \setminus \{p\}$, which is a diffeomorphism, and a constant $c_{p,\epsilon} > 1$ such that
		\begin{equation} \label{e-isolated-point-singularities-uniformly-euclidean}
			c_{p,\epsilon}^{-1} g_0 \leq \Phi_p{}^* \wh g_\epsilon \leq c_{p,\epsilon} g_0.
		\end{equation}
		Here, $g_0$ is the flat Euclidean metric.
	\end{clai*}
	\begin{proof}
		Let $(y^1, y^2, y^3)$ be coordinates near $p$ such that $p = (0, 0, 0)$ and $U_p \approx B_1(0)$. We'll take the inversion $x := |y|^{-2}$ to be our diffeomorphism. By \eqref{e-af-decay-conds} we have $g \leq c' g_0$, so by \eqref{e-isolated-point-singularities-greens} and \eqref{e-inversion-map-1},
		\[ \Phi_p{}^* \wh g_\epsilon = (1 + \epsilon \psi \circ \Phi_p)^4 \Phi_p{}^* g \leq (1 + c \epsilon |x|)^4 \Phi_p{}^* c' g_0 = c'(c\epsilon + |x|^{-1})^4 g_0. \]
		The reverse inequality follows similarly.
	\end{proof}
	
	\begin{clai*}
		For small $\epsilon > 0$, there exists a  surface (possibly disconnected) $T_\epsilon \subset \cup_{p \in \sing g} U_p$ that is minimal with respect to $\wh g_\epsilon$, and such that there is a single component $\Omega'$ of $\Omega \setminus T_\epsilon$ that contains $\Sigma$. Moreover, $\overline{\Omega'} \cap \sing g = \emptyset$, so the metric completion of $(\Omega', g|_{\Omega'})$ is an element of $\Fcirc(\Sigma, \gamma)$.
	\end{clai*}
	\begin{proof}
		By employing the low regularity excision lemma in \cite[Lemma 6.1]{L-M} and standard geometric measure theory we find, for each $p \in \sing g$, a compact $\wh g_\epsilon$-area minimizing element of $[\p U_p] \in H_2(U_p \setminus \{p \}; \Z)$, which does not contain $\{p\}$. Let's denote it $T_{p,\epsilon}$, with $\spt T_\epsilon \subset \overline{U}_p \setminus \{p\}$. From geometric measure theory,
		\begin{equation} \label{e-isolated-points-singularities-regularity}
			\spt T_{p,\epsilon} \setminus \p U_p \text{ is smooth}.
		\end{equation}

		By a straightforward area minimization comparison argument and the locally uniform convergence $\wh g_\epsilon \to g$ on $\overline{U}_p \setminus \{p\}$, we know that
		\begin{equation} \label{e-isolated-singularities-small-area}
			\lim_{\epsilon \to 0} |\spt T_{p,\epsilon}|_{\wh g_\epsilon} = 0,
		\end{equation}
		where $|\cdot|_{\wh g_\epsilon}$ denotes area with respect to $\wh g_\epsilon$.
		
		Denote $W_{p,\tau} := \{ x \in U_p : \dist_g(x; \p U_p) < \tau \}$, where $\tau > 0$ is small. First, we show that $\spt T_{p,\epsilon} \cap (W_{p,2\tau} \setminus W_{p,\tau}) = \emptyset$ provided $\epsilon > 0$ is small (depending on $\tau$). If this were not the case, then for a sequence $\epsilon_j \searrow 0$ we'd be able to choose points $x_{p,j} \in \spt T_{p,\epsilon_j} \cap (W_{p,2\tau} \setminus W_{p,\tau})$ on our area minimizing surfaces. By the fact that $\wh g_{\epsilon_j} \to g$ locally smoothly in $\overline{U}_p \setminus \{p\}$ and from the monotonicity formula (for minimal surfaces) in \emph{small} regions of Riemannian manifolds applied to $\spt T_{p,\epsilon_j} \cap (W_{p,2\tau} \setminus W_{p,\tau})$, we find that
		\[ \liminf_j |\spt T_{p,\epsilon_j}|_{\wh g_\epsilon} > 0, \]
		contradicting \eqref{e-isolated-singularities-small-area}. Thus $\spt T_{p,\epsilon} \cap (W_{p,2\tau} \setminus W_{p,\tau}) = \emptyset$ for small $\epsilon > 0$.
		
		Since $W_{p,2\tau} \setminus W_{p,\tau}$ disconnects $\overline{U}_p$, we have:
		\begin{enumerate}
			\item $T_{p,\epsilon}' \subset \overline{W}_{p,\tau}$, or
			\item $T_{p,\epsilon}' \cap W_{p,2\tau} = \emptyset$,
		\end{enumerate}
		for every connected component $T_{p,\epsilon}' \subset \spt T_{p,\epsilon}$.
		
		Case (2) cannot occur for arbitrarily small $\epsilon > 0$: any nontrivial element of $H_2(\overline{U}_p \setminus \{p\}; \Z)$ that is supported on $\overline{U}_p \setminus W_{p,2\tau}$ has a uniform lower $\wh g_\epsilon$-area bound comparable to the $g$-area of $\p W_{p,2\tau}$ (e.g., compare to the uniformly equivalent setting in $B_1(0) \setminus B_{2\tau}(0)$).
		
		Thus, case (1) will hold for \emph{every} connected component of $\spt T_{p,\epsilon}$ when $\epsilon > 0$ is small, so $\spt T_{p,\epsilon} \cap \p U_p = \emptyset$, so
		\[ \eqref{e-isolated-points-singularities-regularity} \implies \spt T_{p,\epsilon} \text{ is everywhere smooth.} \]
		Since $T_{p,\epsilon}$ generates $H_2(\overline{U}_p \setminus \{p\}; \Z)$, it follows that $T_{p,\epsilon}$ is a multiplicity one embedded submanifold, so the metric completion of $U_p \setminus T_{p,\epsilon}$ does not contain $p$. The claim follows with $T_\epsilon = \cup_{p \in \sing g} \spt T_{p,\epsilon}$.
	\end{proof}
	
	As a consequence of the claim, the metric completion of $(\Omega', g|_{\Omega'})$ is an element of $\Fcirc(\Sigma, \gamma)$, so
	\[ \frac{1}{8\pi} \int_{\p \Omega} H_g \, d\sigma = \frac{1}{8\pi} \int_{\p \Omega'} H_g \, d\sigma \leq \Lcirc(\Sigma, \gamma) = \LL(\Sigma, \gamma), \]
	where the last equality follows from Proposition \ref{p-equality-l-ll}.

	Now we study the rigidity case of \eqref{e-isolated-singularities-inequality}.

	First, we show $R(g) \equiv 0$ on $\Omega \setminus \sing g$. Indeed, whenever $R > 0$ at some point, then one can conformally decrease the scalar curvature, while maintaining scalar nonnegativity, the induced metric on the boundary, and increase the boundary mean curvature pointwise; this would produce 	a singular fill-in of $(\Sigma, \gamma)$ that violates \eqref{e-isolated-singularities-inequality}.
	
	Now we'll show that $\Ric(g) \equiv 0$. Let $h \in C^\infty(\operatorname{Sym}_2 T^* \Omega)$ be an arbitrary smooth symmetric 2-tensor that is supported away from $\Sigma \cup \sing g$. For concreteness, suppose that
	\[ \{ h \neq 0 \} \subset U \subset \subset \Omega  \setminus (\Sigma \cup \sing g). \]
	For small $t$, the symmetric 2-tensors
	\[ g_t := g - th \]
	are 
	{$L^\infty$}
	metrics that are smooth away from $\sing g$. Recall that $R(g) \equiv 0$ on $\Omega \setminus \sing g$, and that $R(g_t) \equiv 0$ near $\sing g$ for all small $t$.
	
	Consider the first eigenvalue, $\lambda_t$, of $-\Delta_{g_t} + \tfrac18 R(g_t)$ on $\Omega$ with Dirichlet boundary conditions: $u_t|_\Sigma = 0$. Note that $\lambda_0 > 0$, so $\lambda_t > 0$ for small $t$. In particular, we can find 
	{$u_t \in W^{1,p}(\Omega)$, $p > 1$ arbitrary, and smooth locally away from $\sing g$,}
	such that
	\begin{equation} \label{e-isolated-singularities-conformal-equation}
		\begin{array}{rl}
		- \Delta_{g_t} u_t + \tfrac18 R(g_t) u_t = 0 & \text{in } \Omega \\
		u_t = 1 & \text{on } \Sigma.
		\end{array}
	\end{equation}
	Note that $u_0 \equiv 1$ and $u_t > 0$ for small $t$. Thus,
	{fixing $p > 3$, by the Sobolev embedding theorem}
	$\wh g_t = u_t^4 g_t$ is a
	{$L^\infty$}
	metric on $\Omega$ for small $t$. We claim that $(\Omega, \wh g_t)$ is a singular fill-in of $(\Sigma, \gamma)$ for small $t$. Indeed, $\sing \wh g_t \subset \sing g$, $\wh g_t|_{\Sigma} = \gamma$, and $R(\wh g_t) \equiv 0$ on $\Omega \setminus \sing g$. Finally,
	\begin{equation} \label{e-isolated-singularities-conformal-mean-curv}
		H_{\wh g_t} = H_g + 4 \tfrac{\p}{\p \nu} u_t.
	\end{equation}
	This is positive for small $t$, since $u_0 \equiv 1$, so, $(\Omega, \wh g_t)$ is a singular fill-in.
	
	Since $H_g$ is assumed to attain equality in \eqref{e-isolated-singularities-inequality}, \eqref{e-isolated-singularities-conformal-mean-curv} implies
	\begin{equation} \label{e-isolated-singularities-mean-curv-ineq-i}
		\int_{\partial \Omega} \tfrac{\p}{\p \nu} u_t \, d\sigma \leq 0
	\end{equation}
	for all small $t$, with equality at $t = 0$. 
	{
	Fix $t \geq 0$ and a smooth background metric $\bar g_t$ with
	\begin{equation} \label{e-isolated-singularities-background-gt}
		\bar c_t^{-1} \bar g_t \leq g_t \leq \bar c_t \bar g_t.
	\end{equation}
	For small $r > 0$, write $U^{(t)}_r$ for the $r$-neighborhood of $\sing g$ with respect to $\bar g_t$ . Recall that $u_t$ is smooth locally away from $\sing g$. The divergence theorem on $\Omega \setminus U^{(t)}_r$, \eqref{e-isolated-singularities-conformal-equation}, and \eqref{e-isolated-singularities-mean-curv-ineq-i} imply
	\[ 0 \geq \frac18 \int_{\Omega \setminus U^{(t)}_r} R(g_t) u_t \, dv_{g_t} + \int_{\partial U^{(t)}_r} \tfrac{\partial}{\partial \nu_{g_t}} u_t \, d\sigma_{g_t}, \]
	where, in the rightmost surface integral, $\nu_{g_t}$ is the normal pointing outside of $U^{(t)}_r$ and $d\sigma_{g_t}$ is the induced area form on $\partial U^{(t)}_r$, both taken with respect to $g_t$. Rearranging and bounding $|\tfrac{\partial}{\partial \nu} u| \leq |\nabla_{g_t} u|$,
	\[ \int_{\Omega \setminus U^{(t)}_r} R(g_t) u_t \, dv_{g_t} \leq 8 \int_{\partial U^{(t)}_r} |\nabla_{g_t} u_t| \, d\sigma_{g_t}. \]
	If $d\sigma_{\bar g_t}$ denotes the area form induced on $\partial U^{(t)}_r$ by $\bar g_t$, then \eqref{e-isolated-singularities-background-gt} implies
	\[ \int_{\Omega \setminus U^{(t)}_r} R(g_t) u_t \, dv_{g_t} \leq \bar c'_t \int_{\partial U^{(t)}_r} |\nabla_{\bar g_t} u_t| \, d\sigma_{\bar g_t}, \]
	for a fixed constant $\bar c'_t$ independent of $r$. We now observe that the left hand side is constant for small enough $r$, since $R(g_t) \equiv 0$ near $\sing g$, and the right hand converges to zero for at least one subsequence $r_i \to 0$ in view of $u_t \in W^{1,p}(\Omega)$. Therefore:
	}
	\[ \int_\Omega R(g_t) u_t \, dv_{g_t} \leq 0 \]
	for all small $t$, with equality at $t=0$. 
	Recall that
	\[ \left[ \frac{d}{dt} R(g_t) \right]_{t=0} = \operatorname{div}_g \operatorname{div}_g h - \Delta_g \operatorname{Tr}_g h + \langle h, \Ric(g) \rangle_g, \]
	that $h \equiv 0$ near $\sing g$, $R(g) \equiv 0$ on $\Omega \setminus \sing g$, and $u_0 \equiv 1$. Thus differentiating the integral of $R(g_t) u_t$ and setting $t = 0$, we find that:
	\[ \int_\Omega \langle h, \Ric(g) \rangle_g \, dv_g \leq 0. \]
	Since we $h$ was arbitrary, we deduce that $\Ric(g) \equiv 0$ on $\Omega \setminus \sing g$. It follows that $\sing g = \emptyset$ by \cite{SmithYang92}. In particular, the result follows from the smooth case: \cite[Theorem 1.4]{M-M}.
\end{proof}

\section{A result in function theory} \label{s-isometry}

The rigidity statements of Theorems \ref{t-LC-1} and \ref{t-LC-2} will be  established using the following characterization of Euclidean balls in terms of their Green's functions, which might be of independent interest; see also \cite[Corollary 2.2]{AgostinianiMazzieri16}:

 \begin{thm} \label{t-isometry}
	 Suppose $(\Omega, g)$ is a compact, connected $n$-dimensional Riemannian manifold, $n \geq 3$, with boundary $\Sigma$ such that:
	 \begin{enumerate}
		 \item There exists an isometric immersion $\iota:  (\Omega, g) \rightarrow \R^n $;
		 \item there exists a  $p \in \Omega \setminus \Sigma$ and a $ \psi : \Omega \setminus \{ p \} \to \R$ such that
		 \[ \Delta_g \psi = 0 \text{ on } \Omega \setminus \{p\}, \; \psi = 0\text{ at } \Sigma, \; \psi = A d^{2-n} + O(1) \]
		 for some $A>0$, where $d = \dist_g(p, \cdot)$; and
		 \item $\psi$ satisfies the boundary condition
		 \begin{equation} \label{e-H-condition-g}
			\tfrac{\p}{\p \nu} \psi + c H = -b \ \mathrm{at} \ \Sigma,
		 \end{equation}
		 where $H$ is the mean curvature of $\Sigma$ with respect to the  unit normal $\nu$ pointing outside of $\Omega$, and $b\ge 0$, $c > 0$ are constants.
	 \end{enumerate}
	 Then $(\Omega,g)$ is isometric to a Euclidean ball in $\R^n$ with radius $r$ so that
	 \[ A(2-n)r^{2-n}+br+c(n-1)=0. \]
 \end{thm}
\begin{proof}
	Without loss of generality, let $\iota(p)= 0$ (the origin of $\R^n$). Let $Q = \iota(\Omega) \subset \R^n$. Let $R$ be the smallest radius so that $\ol B_R(0) \supset Q$. We claim that $\iota(\Sigma) \cap \p B_R(0) \neq \emptyset$. Suppose, by contradiction, that $\p B_R(0) \cap \iota(\Sigma)=\emptyset$. Let $y\in \p B_0(R)\cap Q$, which is nonempty, and let $q\in \ol \Omega \setminus \Sigma$ with $\iota(q)=y$. Since $\iota$ is an isometric immersion, there is an open set $B_{r_{y}}(y) \subset \iota(\Omega)$, for some $r_{y}>0$, because $\iota$ is a local isometry. This contradicts the fact that $\ol B_R(0) \supset Q$.
	
	Let $q_0\in \Sigma$ with $\iota(q_0)=y_0\in \p B_0(R)$. Let $\xi:B_R(0)\to \R$ be
	\[ \xi(y) := A|y|^{2-n}-AR^{2-n}, \]
	a positive harmonic on $B_R(0) \setminus \{0\}$ so that $\xi = 0$ at $\p B_R(0)$. Then $\wt \xi=\xi\circ\iota$ is also positive harmonic on $\Omega \setminus\iota^{-1}(0)$, $\wt\xi \geq 0$ at $\Sigma\setminus \iota^{-1}(o)$, and $\wt\xi \to \infty$ as $q \to \iota^{-1}(o)$. Moreover, $\wt \xi=Ad^{2-n}+O(1)$ near $p$. By the maximum principle we have $\wt\xi\ge \psi$, and hence
	\bee
		\tfrac{\p}{\p \nu} \psi \ge \tfrac{\p}{\p \nu} \wt\xi = A(2-n)R^{1-n}.
	\eee
	On the other hand, by the definition of $R$ as the smallest enclosing radius, $	H(q_0) \ge (n-1) R^{-1}$. By \eqref{e-H-condition-g}, we conclude that
	\[ -b=\tfrac{\p}{\p \nu} \psi + c H\ge  A(2-n)R^{1-n}+c(n-1) R^{-1}. \]
	Hence
	\be\label{e-isometry-1}
		0\ge A(2-n)R^{2-n}+bR+c(n-1).
	\ee
	Next, let $\a$ be a geodesic from $p$ parametrized by arc length defined on $[0,t_\a]$ so that $\a([0,t_\a))\subset \Omega \setminus\Sigma$ and $\a(t_\a)\in \Sigma$. Let
	\[ r=\inf\{t_\a |\ \a \text{ is a geodesic from } p \}>0. \]
	Let $D(r)$ denote the union of all geodesics from $p$ with length $< r$. This is an open set as it is the image of $\exp_p(B_r(0))$ and $\exp_p$ is a local diffeomorphism. We claim that $\iota$ is injective in $\ol D(r)$. If $(\iota(\a)\cap\iota(\beta))\setminus\{o\} \neq\emptyset$, then $\iota(\a)=\iota(\beta)$ because they are straight lines from $o$. But $\iota$ is a local isometry, and thus $\a=\beta$. On the other hand, $\iota$ must be injective on $\a$. So $\iota$ is injective and $\iota(D(r)) = B_r(0)$, $\p \ol D(r) = \p B_r(0)$. Note that $r\le R$.  By the definition of $D(r)$ we conclude that there exists $y_1\in \p B_0(r)$ so that $y_1=\iota(q_1)$ for some $q_1\in \p\Omega$. Let
	\[ \zeta(y)=A|y|^{2-n}-Ar^{2-n}. \]
	Then $\zeta$ is positive harmonic in $B_0(r)$ so that $\zeta=0$ on $\p B_0(r)$. Since $\iota$ is bijective in $D(r)$, we conclude that $\zeta\circ\iota\le \psi$ in $D(r)$. Then as in the above argument we have at $q_1$,
	\be\label{e-isometry-2}
		0\le A(2-n)r^{2-n}+br+c(n-1).
	\ee
	Since $A>0$, $b\ge0$, the function $f(x)=A(2-n)x^{2-n}+bx+c(n-1)$ is strictly increasing on $(0,\infty)$. Since $r\le R$,  by \eqref{e-isometry-1}, \eqref{e-isometry-2}, we conclude that $R=r$ and is the unique positive root of $f(x)=0$ because $c>0$.
		
	It remains to prove $D(r) = \Omega \setminus \Sigma$.  If $\Omega \setminus (\Sigma \cup D(r)) \neq \emptyset$, then there is $q\in (\Omega \setminus \Sigma) \cap \p D(r)$, so $\iota(q)\in \p B(r)$. But $\iota$ is a local isometry, so $\iota(\Omega)\setminus \ol B_0(r) \neq \emptyset$, a contradiction.
\end{proof}

\begin{rem} \label{r-isometry}
	From the proof of the theorem, without assuming $\frac{\p}{\p\nu}\psi+cH \equiv -b$, we still have (in the notation of the proof)
	\begin{align*}
		\sup_{\Sigma} \left( \tfrac{\p}{\p\nu}\psi+cH \right) & \ge A(2-n)R^{1-n}+c(n-1)R^{-1}, \\
		\inf_{\Sigma} \left( \tfrac{\p}{\p\nu}\psi+cH \right) & \le A(2-n)r^{1-n}+c(n-1)r^{-1}.
	\end{align*}
\end{rem}

\section{Proof of Theorems \ref{t-LC-1}, \ref{t-LC-2}, Corollaries \ref{c-capacity-Gauss}, \ref{c-LC-1-vs-szego}, {\ref{c-localization-bray}}} \label{s-LC}

We want to apply the results in section \ref{s-singular} to asymptotically flat manifolds with compact boundary. First, we prove Theorem \ref{t-LC-1}. Arguing as in Section \ref{s-hf-asymptotics}, the metric completion of $(M, \wh g)$, $\wh g := \phi^4 g$, is a smooth manifold which is diffeomorphic to the one-point compactification $\Omega$ of $M$, and carries a metric $\wh g$ which is smooth away from a point $p \in \Omega$ that corresponds to $M$'s infinity.

\begin{lemma} \label{l-compactification}
	The metric $\wh g$ extends across $p$ to a $W^{1,p}$ metric, $p > 6$.
\end{lemma}
\begin{proof}
Using the inversion $y=x/|x|^2$, $p$ will correspond to the origin $y=0$, by \eqref{e-inversion-map-1}, if we let  $h_{ij}=g(\p_{y^i},\p_{y^j})$ and $g_{ij}=g(\p_{x^i},\p_{x^j})$, then
\[ \phi^4h_{ij} = \lf(\frac{\C}{|x|}+O(|x|^{-1-\tau})\ri)^4h_{ij} = \C^4\delta_{ij}+O(|y|^\tau). \]
Hence $\ol g$ is continuous at the origin.  On the other hand,
\bee
\frac{\p}{\p y^p}=\frac{\p x^q}{\p y^p}\frac{\p}{\p x^p}=\lf(\frac{\delta_{pq}}{|y|^2}-\frac{2y^py^q}{|y|^4}\ri)\frac{\p}{\p x^q}.
\eee
Hence
\begin{align*}
\phi^4 h_{ij}
	& = (\C+A)^4\lf( \delta_{ik} -\frac{2y^iy^k}{|y|^2}\ri)\lf( \delta_{jl}-\frac{2y^jy^l}{|y|^2}\ri)(\delta_{kl}+B)\\
	& = (\C+A)^4B\lf( \delta_{ik} -\frac{2y^iy^k}{|y|^2}\ri)\lf( \delta_{jl}-\frac{2y^jy^l}{|y|^2}\ri) +(C+A)^4\delta_{ij}.
\end{align*}
where
\[ |A| + |y| \cdot |\p_yA| + |B| + |y| \cdot |\p_yB|=O(|y|^\tau). \]
Hence $|\p_y (\phi^4 h_{ij})| = O(|y|^{\tau-1})$. The results follows since $\tau > \tfrac12$.
\end{proof}

\begin{proof}[Proof of Theorem \ref{t-LC-1}]
	Recall that we've set
	\[ (\Omega, \wh g) := \text{ metric completion of } (M, \wh g). \]
	By Lemma \ref{l-compactification}, $\wh g$ is a $W^{1,p}$, $p > 6$, metric on $\Omega$ with $\sing \wh g \subset \{p\}$. Since $\Omega$ is diffeomorphic to the one-point compactification of $M$, $\p \Omega \approx \Sigma$. In fact, since $\phi = 1$ on $\Sigma$, $\wh g$ induces the same metric, $\gamma$, on $\p \Omega \approx \Sigma$ as does $g$. The scalar curvature of the conformal metric is
	\[ R(\wh g) = 8 \phi^{-5} (-\Delta_g \phi + \tfrac18 R(g) \phi) = R(g) \phi^{-4} \geq 0 \text{ on } \Omega \setminus \{p\}. \]
	The mean curvature of $\Sigma$ with respect to $\wh g$ and the unit normal $\nu$ pointing outside of $\Omega$ is $H_{\wh g} = -H_g + 4\tfrac{\p}{\p\nu} \phi$ on $\p \Omega$, 	which is \emph{positive} by \eqref{e-LC-1-boundary-assumption}, so  $(\Omega, \wh g)$ is a singular fill-in of $(\Sigma, \gamma)$. 	Thus, from Theorem \ref{t-fsing} or Proposition \ref{p-isolated-singularities-inequality-rigidity} and \eqref{e-capacity-meaning}:
	\begin{multline*}
		\LL(\Sigma, \gamma) \geq \frac{1}{8\pi} \int_{\p \Omega} H_{\wh g} \, d\sigma = \frac{1}{8\pi} \int_{\p \Omega} (-H_g + 4 \tfrac{\p}{\p \nu} \phi) \, d\sigma \\
			= - \frac{1}{8\pi} \int_{\p \Omega} H_g \, d\sigma + 2 \, \CgSM,
	\end{multline*}
	which gives \eqref{e-LC-1}, and thus (1).

	If equality holds in \eqref{e-LC-1}, then Propositions \ref{p-fsing-rigidity} or \ref{p-isolated-singularities-inequality-rigidity} imply that $(\Omega, \wh g)$ is a mean-convex handlebody with a flat metric.
	
	Finally, suppose $\Sigma$ is a topological sphere. First, assume that equality holds in \eqref{e-LC-1}. Then, by the previous paragraph, $(\Omega, \wh g)$ is a genus-0 mean-convex flat handlebody, i.e., a 3-dimensional ball. As such, it can be isometrically immersed in $(\R^3, g_0)$.

	Let $\psi := \phi^{-1}-1$. Note that, by the formula for the transformation of the Laplacian among the conformal metrics $\wh g = \phi^4 g$:
	\[ \Delta_{\wh g} \psi = \phi^{-5} (\Delta_{g} (\phi \psi) - \psi \Delta_g \phi) = 0. \]
	Moreover,
	\[ \psi = 0 \text{ on } \p \Omega, \]
	\[ \psi = \CgSM \dist_{\wh g}(\cdot, p)^{-1} + O(1) \text{ near } p, \]
	the latter holding in view of Lemma \ref{l-phi}. Finally,
	\begin{equation} \label{e-LC-1-blowdown-H}
		H_{\wh g} + 4 \tfrac{\p}{\p \nu} \psi = -H_g \text{ on } \p \Omega.
	\end{equation}
	Since $H_g \geq 0$ is a constant, it follows from our characterization of Euclidean balls in Theorem \ref{t-isometry} that $(\Omega, \wh g)$ is isometric to a standard round ball in $\R^3$ of radius $r>0$ for some $r>0$.  Hence $(M,g)$ is isometric to $(B_r(0) \setminus \{0\},(\psi+1)^4g_0)$ where $g_0$ is the standard Euclidean metric. Using the Kelvin transform, one concludes that $(M,g)$ is Riemannian Schwarzschild outside of a coordinate sphere, as asserted.

	For the converse implication, see the explicit computation in Subsection \ref{s-schwarzschild-computations}, which shows that the left and right hand sides of \eqref{e-LC-1} coincide in the case of Schwarzschild and coordinate spheres $\Sigma$.
\end{proof}

\begin{proof}[Proof of Corollary \ref{c-capacity-Gauss}]
	This is Theorem \ref{t-LC-1} with $H_g = 0$ on $\Sigma$.
\end{proof}

\begin{proof}[Proof of Corollary \ref{c-LC-1-vs-szego}]
	Since $H_g < 4 |\nabla \phi|$ on $\Sigma$,
	\[ 2 \, \mathcal{C}_g(\Sigma, M \setminus \Omega) \leq \LL(\Sigma, \gamma) + \frac{1}{8\pi} \int_\Sigma H_g \, d\sigma \]
	by Theorem \ref{t-LC-1}. The second term of the right hand side is $\leq$ the first by $H_g > 0$ and Definition \ref{d-f-l}. The result follows from \eqref{e-ll-h-gauss}.
	
	For the rigidity case, Theorem \ref{t-LC-1} (2) and Theorem \ref{t-fsing} imply that the blowdown $(M \setminus \Omega, \widehat{g})$ is Euclidean. On the other hand, so is the fill-in $(\Omega, g)$. We thus have $H_g = H_{\widehat{g}}$ on $\Sigma$ by the Cohn--Vossen theorem (see \cite{Nirenberg53}). Plugging into \eqref{e-LC-1-blowdown-H} we get $H_{\widehat{g}} + 2 \tfrac{\p}{\p \nu} \psi = 0$ on $\p \Omega$. It follows from Theorem \ref{t-isometry} that $(\Omega, g)$ and $(M \setminus \Omega, \widehat{g})$ are geodesic balls in $(\R^3, g_0)$.
\end{proof}

\begin{proof}[Proof of Theorem \ref{t-LC-2}]
	We begin by applying the idea to reflect $M$ through its boundary, as in Bunting--Masood-ul-Alam \cite{BuntingMasood87}. Let
	\[ w := 1-\phi \text{ on } M. \]
	Let $(\wt M, \wt g)$ be the mirror image of $(M,g)$ and
	\[ (N,h) := (M,g) \sqcup (\wt M, \wt g) / \sim, \]
	where $\sim$ denotes the identification of $\Sigma \subset M$ with its mirror image in $\wt M$. (Note that $h$ is only Lipschitz across $\Sigma$.) We denote the infinities coming from $M$, $\wt M$ as $\infty_M$, $\infty_{\wt M}$, respectively. Let $\wt w$ be the harmonic function which is the mirror image of $w$, i.e., near $\Sigma$, we have $\wt w(t,x')=-w(-t,x')$ in exponential normal coordinates $(t, x')$ over $\Sigma$, with $t = \dist_g(\Sigma, \cdot)$. Define $u : N \to (0, 1)$ to be
	\bee
		u(x) :=
			\begin{cases}
				\tfrac12 (1+w) \text{ on } M, \\
				\tfrac12 (1+\wt w) \text{ on } \wt M.
			\end{cases}
	\eee
	Note that $u(x) \to 1$ as $x \to \infty_M$, while $u(x) \to 0$ as $x \to \infty_{\wt M}$. Since $w=\wt w=0$ on $\Sigma$, $(N, u^4 h)$ is smooth up to (but not across) $\Sigma$, Lipschitz across $\Sigma$, and has scalar curvature
	\[ R(u^4 h) = 8 u^{-5} (-\Delta_h u + \tfrac18 R(h) u) \geq 0 \text{ on } N \setminus \Sigma. \]
	The mean curvatures induced by $u^4 g$, $u^4 \wt g$ on the two sides of $\Sigma$, with respect to the unit normal $\eta$ pointing toward $\infty_M$, are
	\begin{equation} \label{e-LC-2-H-plus}
		H_{\Sigma,u^4 g} = u^{-2} (H_{\Sigma,g} + 4 \tfrac{\p}{\p \eta} \log u) = 4 (H_{\Sigma,g} - 4 \tfrac{\p}{\p \eta} \phi),
	\end{equation}
	\begin{equation} \label{e-LC-2-H-minus}
		H_{\Sigma,u^4 \wt g} = u^{-2} (-H_{\Sigma,g} + 4 \tfrac{\p}{\p \eta} \log \wt u) = 4 (-H_{\Sigma,g} - 4 \tfrac{\p}{\p \eta} \phi).
	\end{equation}
	Since $H_{\Sigma,g} \leq 0$ with respect to $\eta$,
	\begin{equation} \label{e-LC-2-H-jump}
		\eqref{e-LC-2-H-plus}-\eqref{e-LC-2-H-minus} \implies H_{\Sigma,u^4g} \leq H_{\Sigma,u^4 \wt g}.
	\end{equation}
	
	With $c$ as in the theorem statement, let $N(c) := \{ u \leq c\} \subset N$. Note that Theorem \ref{t-fsing} applies to the metric completion of $(u^4h, N(c))$:	
	\begin{enumerate}
		\item By \eqref{e-LC-2-H-jump}, we have a $(\mathbf{C})(\mathbf{a})$ singularity on $\Sigma$.
		\item By the asymptotic behavior of $\phi$ (and $u$) from Lemma \ref{l-phi}, $u^4 h$ can be compactified by adding a point $p_\infty$ that corresponds to $\infty_{\wt M}$, at which $u^4 h$ has a $(\mathbf{C})(\mathbf{b})$ singularity; see Lemma \ref{l-compactification}.
		\item By the formula for the change in mean curvatures under conformal metric changes and assumption \eqref{e-LC-2-boundary-assumption},
		\begin{equation} \label{e-LC-2-H-u4g}
		H_{\Sigma(c),u^4 g} = c^{-2} ( H_{\Sigma(c),g} - 2 c^{-1} \tfrac{\p}{\p \eta_c} \phi)
		\end{equation}
		is \emph{positive}, where $H_{\Sigma(c),g}$ is also computed with respect to the normal $\eta_c$ pointing toward $\infty_M$.
	\end{enumerate}

	By the scaling properties of the total boundary mean curvature, and $\p N(c) = \Sigma(c)$, we see that Theorem \ref{t-fsing} implies:
	\begin{multline} \label{e-LC-2-LL-H}
		c^2 \LL(\Sigma(c), \gamma_c) = \Lambda(\Sigma(c), c^4\gamma_c) \ge \frac1{8\pi}\int_{\Sigma(c)} H_{\Sigma(c),u^4 g} \, c^4 d\sigma_c \\
		\implies \LL(\Sigma(c), \gamma_c) \geq \frac{c^2}{8\pi} \int_{\Sigma(c)} H_{\Sigma(c), u^4 g} \, d\sigma_c.
	\end{multline}

Together, \eqref{e-LC-2-H-u4g}, \eqref{e-LC-2-LL-H}, imply:
	\[ \LL(\Sigma(c), \gamma_c) \geq \frac{1}{8\pi} \int_{\Sigma(c)} H_{\Sigma(c),g} \, d\sigma_c - \frac{c^{-1}}{4\pi} \int_{\Sigma(c)} \tfrac{\p}{\p \eta_c} \phi \, d\sigma_c. \]
	We can integrate by parts using \eqref{e-capacity-pde},  \eqref{e-capacity-meaning} to get:
	\[ \LL(\Sigma(c), \gamma_c) \geq \frac{1}{8\pi} \int_{\Sigma(c)} H_{\Sigma(c),g} \, d\sigma_c + c^{-1} \, \CgSM, \]
	which is precisely \eqref{e-LC-2}.

	Now suppose that equality holds in \eqref{e-LC-2}. Then equality also holds in \eqref{e-LC-2-LL-H}, and thus in the invocation of Theorem \ref{t-fsing} to $(u^4h, N(c)\cup \{p_\infty\})$. Thus, by Proposition \ref{p-fsing-rigidity}, $u^4h$ is smooth and flat across $\infty_{\wt M}$, though perhaps not across $\Sigma$. Nonetheless, $H_{\Sigma,u^4g}= H_{\Sigma,u^4\wt g}$ along $\Sigma$, so
	\begin{equation} \label{e-LC-2-H-zero}
		H_{\Sigma,g} \equiv 0 \text{ on } \Sigma,
	\end{equation}
	by virtue of \eqref{e-LC-2-H-jump}.
	
	Recall that, for the rigidity case, we're additionally assuming $\Sigma(c)$ is a sphere of positive Gauss and mean curvature.

	Since equality holds in \eqref{e-LC-2-LL-H} and $(\Sigma(c), c^4 \gamma_c)$ has positive Gauss curvature, the results and methods in \cite{ShiTam02} allow one to attach $(N(c), u^4 h)$ to the exterior of a convex set in $\R^3$, along $(\Sigma(c), c^4 \gamma_c)$, to obtain a smooth 3-manifold $\wt N$ with piecewise smooth metric $\mathfrak{g}$. By the monotonicity formula in \cite{ShiTam02}, we conclude that the manifold has only one end and zero ADM mass. By \cite[Theorem 2]{McFeronSzekelyhidi12}, $(\wt N, \mathfrak{g})$ is $C^{1,\alpha}$ isometric to $(\R^3, g_0)$, i.e., there is a $C^{1,\alpha}$ diffeomorphism $\Phi : \wt N \to \R^3$ so that $\Phi : (\wt N, \mathfrak{g}) \to (\R^3, g_0)$ is a Riemannian isometry. Moreover, $\Phi$ is smooth away from $\Sigma$ by Myers--Steenrod \cite{MyersSteenrod39}.
		
	We may assume $\Phi(\infty_{\wt M}) = 0$ (the origin). Note that $\tfrac{\p}{\p \eta} \phi < 0$ on $\Sigma$. For $t>0$ small enough, $\{\phi= 1-t\}$ is a smooth graph over $\Sigma$ contained in a small neighborhood of $\Sigma$. Observe that for any $\eps>0$, there is $t_0 \in (0,1)$ such that if $0<t<t_0$, $t \in (0, t_0) \implies |H_g(t)| \leq \eps$,	where $H_g(t)$ is the mean curvature of $\{\phi= 1-t\}$; here we've used \eqref{e-LC-2-H-zero}. Let $D(t)$ be the domain in $\wt N$ bounded by $\{\phi = 1-t\}$.

	If $R(t)$ is the radius of the smallest ball centered at the origin and which contains $\Phi(D(t))$, then by Remark \ref{r-isometry} we have
	\[ \eps \ge -\tfrac12 \, \CgSM  (R(t))^{-2}+\tfrac12R(t) \\
		\implies 2 \eps R(t)^{2}+ \CgSM \ge R(t). \]
	Conversely, let $r(t)$ be the radius of the largest ball centered at the origin and contained in $\Phi(D(t))$. By Remark \ref{r-isometry}, we similarly have
	\[ r(t)\ge -2\e r(t)^{2}+ \CgSM. \]
	Since $R(t)$, $r(t)$ are uniformly bounded, we may send $\eps \to 0$ and conclude that $\Phi(\wt M)$ is a ball of radius $\CgSM$. As before we conclude that $(M,g)$ is Riemannian Schwarzschild outside the horizon.
\end{proof}

{
\begin{proof}[Proof of Corollary \ref{c-localization-bray}]
	The fact that the set of $c$ for which $\Sigma_c$ is a sphere of positive Gauss and mean curvatures contains an interval of the form $[c_0, 1)$, $c_0 \in (\tfrac12, 1)$, follows from Lemma \ref{l-phi}. By \eqref{e-ll-h-gauss},
	\[ \LL(\Sigma_c, \gamma_c) - \frac{1}{8\pi} \int_{\Sigma_c} H_g \, d\sigma_c = \frac{1}{8\pi} \int_{\Sigma_c} (H_{g_0} - H_g) \, d\sigma_c, \]
	where $H_{g_0}$ denotes the mean curvature of $\Sigma_c$ when embedded, with its induced metric $\gamma_c$, into $(\R^3, g_0)$.
\end{proof}
}

\section{Results in Schwarzschild manifolds and some examples} \label{s-schwarzschild}

\subsection{Some computations} \label{s-schwarzschild-computations}

Let $(M_m^3, g_m)$, $m > 0$, be a (doubled) Riemannian Schwarzschild manifold:
\[ M_m^3  = \R^3 \setminus \{0 \}, \; g_m = \left( 1 + \frac{m}{2r} \right)^4 g_0. \]
Here, $g_0$ is the Euclidean metric on $\R^3 \setminus \{0\}$, and $r = r(x) := |x|$. Consider $\Sigma = \{ r = r_0 \}$, $M_m^\Sigma = \{ r \geq r_0 \}$, for $r_0 \in (0, \infty)$. The metric $\gamma$ induced on $\Sigma$ is that of a round sphere of radius:
\begin{equation} \label{e-schwarzschild-sphere-area-radius}
	r_A(\Sigma) = \left( 1 + \frac{m}{2r_0} \right)^2 r_0.
\end{equation}
In the embedding $(\Sigma, \gamma) \hookrightarrow (\R^3, g_0)$ we have mean curvature
\begin{equation} \label{e-schwarzschild-sphere-H0}
	H_{g_0} = 2 \left( 1 + \frac{m}{2r_0} \right)^{-2} r_0^{-1}.
\end{equation}
The mean curvature $H_{g_m}$ of $\Sigma$ in Schwarzschild with respect to the unit normal pointing to $\infty$ is found by a conformal computation to be
\begin{equation} \label{e-schwarzschild-sphere-H}
	H_{g_m} = 2 \left( 1 + \frac{m}{2r_0} \right)^{-3} \left( 1 - \frac{m}{2r_0} \right) r_0^{-1}.
\end{equation}
Therefore,
\begin{equation} \label{e-schwarzschild-sphere-mby}
	\frac{1}{8\pi} \int_\Sigma H_{g_0} \, d\sigma + \frac1{8\pi} \int_\Sigma H_{g_m} \, d\sigma =  2 \,  r_0 + m.
\end{equation}
The boundary capacity potential for $\Sigma \subset (M_m^\Sigma, g_m)$ is:
\begin{equation} \label{e-schwarzschild-sphere-phi}
	\phi = \frac{2r_0}{m} \left( 1 + \frac{m}{2r_0} \right) \left[ 1 - \left( 1 + \frac{m}{2r} \right)^{-1} \right] = \left( r_0+\frac m2 \right)r^{-1}+O(r^{-2}).
\end{equation}
Thus,
\begin{equation} \label{e-schwarzschild-sphere-capacity}
	\mathcal{C}_{g_m}(\Sigma, M_m^\Sigma) = r_0 + \frac{m}{2}.
\end{equation}
Thus   in terms of $r_A(\Sigma)$ from \eqref{e-schwarzschild-sphere-area-radius} by solving for $r_0$ in terms of $r^{\operatorname{Area}}_0$:
\begin{equation} \label{e-schwarzschild-sphere-capacity-area-radius}
	\mathcal{C}_{g_m}(\Sigma, M_m^\Sigma) = \frac{r_A(\Sigma)}{2} \left( 1 \pm \sqrt{1 - \frac{2m}{r_A(\Sigma)}} \right),
\end{equation}
where $\pm$ is a ``$+$'' when $r > \tfrac12 m$, and a ``$-$'' when $r < \tfrac12 m$. In any case, by \eqref{e-schwarzschild-sphere-mby}, \eqref{e-schwarzschild-sphere-capacity}:
\begin{equation} \label{e-schwarzschild-capacity-identity-new}
	\LL(\Sigma, \gamma) + \frac{1}{8\pi} \int_{\Sigma} H_{g_m} \, d\sigma \equiv 2 \, \mathcal{C}_{g_m}(\Sigma, M_m^\Sigma)
\end{equation}
no matter what $r_0 \in (0, \infty)$ was to begin with.

\begin{rem} \label{r-negative-schwarzschild}
	The computations above also work for $m = 0$ and $m < 0$ with $M_m := \R^3$, $M_m := \R^3 \setminus \{ r \leq \tfrac12 |m| \}$, respectively.
\end{rem}

\subsection{Proofs of Theorem \ref{t-szego-schwarzschild} and its corollaries }

\begin{proof}[Proof of Theorem \ref{t-szego-schwarzschild}]
	We symmetrize as in the proof of the Poincar\'e--Faber--Szeg\"o inequality \eqref{e-szego} (see \cite{PolyaSzego51}), but we will use H. Bray's isoperimetric inequality in $(M_m, g_m)$ \cite{Bray01} instead of the classical isoperimetric inequality in $(\R^3, g)$. In what follows, for any surface $\Sigma$ homologous to the horizon $\Sigma_H \subset (M_m, g_m)$ we let $|\Sigma|$ denote its area and $V(\Sigma)$ denote the signed volume bounded between $\Sigma$, $\Sigma_H$.
	
	Let $\phi$ be the harmonic potential for $\mathcal{C}_{g_m}(\Sigma, M_m^\Sigma)$, extended as $\phi \equiv 1$. For $t \in (0, 1]$, let $\Omega_t := \{ \phi \geq t \}$, let $\Sigma_t := \p \Omega_t$, and $V(t) := V(\Sigma_t)$. It is not hard to see that
	\begin{equation} \label{e-szego-schwarzschild-sigma-t}
		\Sigma_t \text{ is an embedded surface, } V'(t) = \int_{\Sigma_t} \frac{d\sigma_t}{|\nabla \phi|} \text{ for a.e. } t \in (0,1),
	\end{equation}
	where $d\sigma_t$ is the area element of $\Sigma_t$. Now let $\Sigma_t^*$ denote the unique rotationally symmetric sphere enclosing signed volume $V(t)$ with $\Sigma_H$; denote by $\Omega_t^*$ the corresponding annular region with $\p \Omega_t^* = \Sigma_t^*$ and signed volume $V(t)$. Also define $\phi^* : M_m \to [0,1]$ by
	\[ \phi^* = t \text{ on } \Sigma_t^* \text{ for } t \in (0, 1], \; \phi^* \equiv 1 \text{ on } \Omega_1^*. \]
	It is not hard to see that $\phi^*$ is Lipschitz, and
	\begin{equation} \label{e-szego-schwarzschild-comparison}
		|\nabla \phi^*| = \text{const on } \Sigma_t^*, \; V'(t) = \int_{\Sigma_t^*} \frac{d\sigma_t^*}{|\nabla \phi^*|} \text{ for } t \in (0, 1),
	\end{equation}
	where $d\sigma_t^*$ is the area element of $\Sigma_t^*$. By the coarea formula, \eqref{e-szego-schwarzschild-comparison}, the isoperimetric inequality, and Cauchy--Schwarz, we have
	\begin{multline*}
		\int_{M_m} |\nabla \phi^*|^2 \, dv_{g_m} = \int_0^1 \left( \int_{\Sigma_t^*} |\nabla \phi^*| \, d\sigma_t^* \right) \, dt = \int_0^1 \frac{|\Sigma_t^*|^2}{V'(t)} \, dt \\
		\leq \int_0^1 \frac{|\Sigma_t|^2}{V'(t)} \, dt \leq \int_0^1 \left( \int_{\Sigma_t} |\nabla \phi| \, d\sigma_t \right) \, dt = \int_{M_m} |\nabla \phi|^2 \, dv_{g_m}.
	\end{multline*}
	The right hand side is $\mathcal{C}_{g_m}(\Sigma, M_m^\Sigma)$ by \eqref{e-capacity-meaning}, while the left hand side is an upper bound for $\mathcal{C}_{g_m}(\Sigma^*, M_m^{\Sigma^*})$ by the variational characterization of capacity (see \cite[Section 6]{Bray01}).
\end{proof}

Combined with the  computation of capacity on rotationally symmetric spheres in Schwarzschild   and \eqref{e-braymiao-1}, \eqref{e-szego-schwarzschild} gives upper and lower bound of $\mathcal{C}_{g_m}(\Sigma, M_M^\Sigma)$ :
\begin{cor}\label{c-capacity-ul-bound}
	We have:
\bee\label{e-schwarzschild-capacity-upper-lower-bound}
	  \frac{r_{A}(\Sigma^*)}{2} \left( 1 + \sqrt{1 - \frac{2 \, \mhawk(\Sigma^*)}{r_{\operatorname{A}}(\Sigma^*)}} \right)  \leq \mathcal{C}_{g_m}(\Sigma, M_M^\Sigma)  \leq \frac{r_A (\Sigma)}{2} \left( 1 + \sqrt{1 - \frac{2 \, \mhawk(\Sigma)}{r_{\operatorname{A}}(\Sigma)}} \right),
\eee
where $r_A(\Sigma')$ is the area radius of a surface $\Sigma'$, i.e., $4\pi (r_A(\Sigma'))^2=|\Sigma'|$ which is the area of $\Sigma'$, and $\mhawk(\cdot)$ denotes the Hawking quasi-local mass \cite{Hawking68},
\[ \mhawk(\Sigma') := \sqrt{\frac{|\Sigma'|}{16\pi}} \left( 1 - \frac{1}{16\pi} \int_\Sigma H_{g_m}^2 \, d\sigma \right), \]
with $d\sigma$ the induced area element of $\Sigma' \subset (M_m, g_m)$. Moreover, if either inequality is an equality, then they are both equalities, and $\Sigma = \Sigma^*$.
\end{cor}

\begin{cor} \label{c-schwarzschild-comparison}
	Let $\Sigma$ be a sphere of positive Gauss curvature bounding a domain containing with the horizon $\Sigma_H$ in a Schwarzschild manifold $(M_m, g_m)$ with mass $m > 0$, so that as a surface in $\R^3$, $\Sigma$ contains the origin. Suppose that the mean curvature with respect to the unit normal pointing into $M^\Sigma$:
	\begin{equation} \label{e-schwarzschild-comparison-relatively-trapped}
		H_{g_m} < 4 |\nabla \phi| \text{ on } \Sigma;
	\end{equation}
	here, $\phi$ is the boundary capacity potential of $\Sigma$ that corresponds to $\mathcal{C}_{g_m}(\Sigma, M^\Sigma)$ per \eqref{e-capacity-pde}.  If $\Sigma^*$ is the unique rotationally symmetric sphere in $(M_m, g_m)$ that bounds a domain with $ \Sh$ with the same (signed) volume as $\Sigma$, then
	\be \label{e-schwarzschild-comparison}
		\int_{\Sigma^*} (H_{g_0} + H_{g_m}) \, d\sigma^* \leq \int_{\Sigma} (H_{g_0} + H_{g_m}) \, d\sigma.
	\ee
	Moreover,  equality holds in \eqref{e-schwarzschild-comparison} if and only if $\Sigma = \Sigma^*$.
\end{cor}
\begin{proof}
	From Theorems \ref{t-LC-1}, \ref{t-szego-schwarzschild}, and \eqref{e-schwarzschild-capacity-identity-new}:
	\begin{align*}
		\frac{1}{8\pi} \int_{\Sigma} (H_{g_0} + H_{g_m}) \, d\sigma
			& \geq 2 \, \mathcal{C}_{g_m}(\Sigma, M_m^\Sigma) \\
			& \geq 2 \, \mathcal{C}_{g_m}(\Sigma^*, M_m^{\Sigma^*}) = \frac{1}{8\pi} \int_{\Sigma^*} (H_{g_0} + H_{g_m}) \, d\sigma.
	\end{align*}
	The first result follows. Rigidity follows from Theorem \ref{t-szego-schwarzschild}.
\end{proof}

 Another  corollary of Theorems \ref{t-LC-2} and \ref{t-szego-schwarzschild} is the following:

\begin{cor} \label{c-schwarzschild-mby-comparison}
	Let $\Sigma$ be a sphere of positive Gauss curvature bounding a domain  with the horizon $\Sigma_H$ in a Schwarzschild manifold $(M_m, g_m)$ with mass $m > 0$, and which is weakly outer trapped (its mean curvature vector points to $r=\infty$) and contains the origin in $\R^3$. Let $\phi$ be the boundary capacity potential of $\mathcal{C}_{g_m}(\Sigma, M_m^\Sigma)$ given by \eqref{e-capacity-pde}. Set
	\[ u := \tfrac12 (2-\phi), \; \Sigma_c := \{ u = c \}, \; c \in (\tfrac12, 1). \]
	Suppose that $c \in (\tfrac12, 1)$ is a regular value of $u$ and that the equipotential surface $\Sigma_c$ has positive Gauss curvature and its mean curvature with respect to the unit normal that points into $M_\Sigma$ satisfies
	\begin{equation} \label{e-schwarzschild-mby-comparison-relatively-untrapped}
		H_{g_m} > - 4 |\nabla \log u| \text{ on } \Sigma_c.
	\end{equation}
	If $\Sigma_c$ bounds the same (signed) volume with $\Sigma_H$ in $M_m$ the rotationally symmetric sphere $\Sigma_c^*$, then
	\begin{equation} \label{e-schwarzschild-mby-comparison}
		\mby(\Sigma_c) \geq (c^{-1} - 1) \left( \frac{\mby(\Sigma_c^*) }{m} - 1 \right)^{-1}  \mby(\Sigma_c^*),
	\end{equation}
	and equality holds if and only if $\Sigma$ is the horizon and $\Sigma_c = \Sigma_c^*$. Here $\mby(\wt\Sigma) := \frac{1}{8\pi} \int_{\wt\Sigma} (H_{g_0} - H_{g_m}) \, d\wt\sigma$, for $\wt\Sigma = \Sigma_c$, $\Sigma_c^*$, respectively.
\end{cor}

\begin{proof}
	This follows like Corollary \ref{c-schwarzschild-comparison} above, except we evaluate the Brown--York mass of rotationally symmetric spheres using  \eqref{e-schwarzschild-sphere-mby}.
\end{proof}

\subsection{Comparing the ADM mass to the  $\LL$-invariant} \label{s-examples-inequality-1}

We have been discussing three quantities on an asymptotically flat manifold $(M^3,g)$ with nonnegative scalar curvature with compact boundary $\Sigma$ which is a horizon. Namely, the ADM mass $\mathfrak{m}_{\text{ADM}}$, the capacity $\mathcal{C}(M,\Sigma)$ and the quantity $\LL(\Sigma,\gamma)$ where $\gamma$ is the induced metric. Bray's capacity inequality \eqref{e-m-cap} and Theorem \ref{t-LC-1} imply that
\begin{equation}\label{e-ex-1}
  \begin{array}{r}
   \mathfrak{m}_{\text{ADM}}\ge \CgSM; \\
  \tfrac12 \LL(\Sigma,\gamma)\ge \CgSM.
  \end{array}
\end{equation}
In this section we'll point out how there is no a priori relationship between the ADM mass of a complete, asymptotically flat manifold and the $\LL$-invariant of its boundary (recall Definition \ref{d-ff-ll}).

Fist, let $m > m' > 0$ and consider the Schwarzschild metrics
\[ g_m := \left( 1 + \frac{m}{2r} \right)^4 g_0, \; g_{m'} := \left( 1 + \frac{m'}{2r} \right)^4 g_0, \]
taken to be defined on $M_m := \R^3 \setminus B_{m/2}(0)$, $M_{m'} := \R^3 \setminus \{0\}$, respectively. We can glue these two manifolds together by identifying the horizon $\Sigma_H = \{ |x| = m/2 \} \subset (M_m, g_m)$  with the coordinate sphere $\Sigma' = \{ |x| = r' \} \subset (M_{m'}, g_{m'})$, where the latter is isometric to $\Sigma_H$. This is doable because there is a unique such $r' \in (\tfrac{m'}{2}, \infty)$, given by
\[ \left( 1 + \frac{m'}{2r'} \right)^4 4\pi (r')^2 = 4\pi (2m)^2. \]
Define
\[ (\wt M, \wt g) := (M_m, g_m) \sqcup (\{ m'/2 \leq |x| \leq r' \} \subset M_{m'}, g_{m'}) / \sim \]
where the symbol $/ \sim$ on the right denotes the boundary identification $\Sigma_H \cong \Sigma'$. This is a manifold with a Lipschitz metric which is smooth and scalar flat away from $\Sigma_H$, and has minimal boundary $\Sigma_H'$. By \eqref{e-schwarzschild-sphere-H0}, we see that
\begin{equation} \label{e-right-H-jump-ADM-mby}
	\madm(\wt M, \wt g) = m > m' = \frac12 \cdot \frac{1}{8\pi} \int_{\Sigma_H'} H_{g_0}   \, d\sigma= \tfrac12 \LL(\Sigma_H',\gamma').
\end{equation}
Here, $H_{g_0}$ denotes the mean curvature of $\Sigma_H'$ in $\R^3$. Even though $\wt g$ is not smooth on $\Sigma'$, the mean curvature jump on it is such that one may apply the smoothing method of \cite[Proposition 3.1]{Miao02} to the (larger) rotationally symmetric manifold
\[ (M_m, g_m) \sqcup (\{ 0 < |x| \leq r' \} \subset (M_{m'}, g_{m'})) / \sim, \]
where $\sim$ still denotes the identification $\Sigma_H \cong \Sigma'$. One obtains rotationally symmetric $(\wt M_i, \wt g_i)$ that converge smoothly away from $\Sigma_H$ to the original nonsmooth manifold, and, by \cite[Lemma 4.2]{Miao02},
\[ \lim_i \madm(\wt M_i, \wt g_i) = \madm(\wt M, \wt g); \]
on the left hand side, $\madm$ is computed on the end coming from $M_m$. By a simple mean convex barrier argument, $(\wt M_i, \wt g_i)$ contains a horizon $(\Sigma_{H,i}', \gamma_i')$ whose metric converges smoothly to that of $\Sigma_H'$. Thus
\[ \lim_i \int_{\Sigma_{H,i}'} H_{0,i} \, d\sigma_{\gamma_i} = \int_{\Sigma_H'} H_{g_0} \, d\sigma_{\gamma'}. \]
Here, $H_{0,i}$ denotes the mean curvature of the isometric embedding of either $\Sigma_{H,i}'$ or $\Sigma_H'$ into $\R^3$. The last two limits allow us to smooth the manifold giving the counterexample in \eqref{e-right-H-jump-ADM-mby}.

\begin{rem}
	One can arrange for  \eqref{e-right-H-jump-ADM-mby} to hold with $\gg$ in the inequality by suitably picking $m \to \infty$, $m' \to 0$.
\end{rem}

Conversely, \cite{MS15} shows the existence of asymptotically flat extensions of horizons $(\Sigma, \gamma)$ with positive Gauss curvature, with ADM mass arbitrarily close to $\sqrt{|\Sigma|_\gamma/16\pi}$. The classical Minkowski inequality on $\R^3$ and \cite[Theorem 1.5]{M-M} imply, for nonround $\Sigma$, that
\begin{equation} \label{e-mby-minkowski-inequality}
	\sqrt{\frac{|\Sigma|_\gamma}{16\pi}} < \frac{1}{16\pi} \int_\Sigma H_{g_0} \, d\sigma= \tfrac12 \LL(\Sigma, \gamma),
\end{equation}
where $H_{g_0}$ is the mean curvature of $\Sigma$ when embedded in $(\R^3, g_0)$.
	
\begin{rem}
	One can arrange for \eqref{e-mby-minkowski-inequality} to hold with $\ll$ in the inequality by suitably picking $(\Sigma, \gamma)$ to be a long, thin tube.
\end{rem}

\appendix

\section{Asymptotically flat manifolds and capacity} \label{s-af-implications}

We recall the definition of an asymptotically flat manifold.

\begin{df}\label{d-AF}
	A Riemannian $3$-manifold $(M, g)$ is called asymptotically flat (of order $\tau$) if there is a compact set $K $ such that $M \setminus K$ is diffeomorphic to $\R^3$ minus a ball and, with respect to the standard coordinates on $\R^3$, the metric $g$ satisfies
	\be \label{e-af-decay-conds}
		g_{ij} = \delta_{ij} + O ( | x |^{-\tau} ) , \ \p g_{ij}  =  O ( | x |^{ - \tau - 1}) , \ \p \p g_{ij}  =  O ( | x |^{ - \tau - 2 } )
	\ee
	for some $\tau \in (\tfrac12, 1]$, and the scalar curvature $ R(g)$ satisfies
	\be \label{e-af-decay-conds-scalar-curv}
		R(g) = O (| x |^{-q})
	\ee
	for some $q > 3$. Here $\partial$ denotes  partial differentiation on $\R^3$.
\end{df}

In asymptotically flat manifolds of order $\tau$, we have precise control of the boundary capacity potential from Definition \ref{d-capacity},  \eqref{e-capacity-meaning}. Namely:

\begin{lemma} \label{l-phi}
	If $(M^3, g)$ is asymptotically flat of order $\frac12<\tau<1$, then in Euclidean coordinates $x$ near infinity,
	\begin{equation} \label{e-capacity-expansion}
		\phi(x)=\frac{\CgSM}{|x|}+O_2(|x|^{-1-\tau}),
	\end{equation}
	for a constant $\CgSM > 0$.
\end{lemma}
\begin{proof}
	We'll write $\Delta_0$ for the Euclidean Laplacian and use the notation $u_i = \partial_i u$, etc.
Here and below, $C$ will denote positive constant which is independent of $x$ and $r$.
Let $\psi=ar^{-1}- r^{-1-\e}$, with $a,   \e>0$. Then
		\be\label{e-laplacian-expansion}
			\Delta_g \psi = \Delta_0 \psi +\sigma_{ij}\psi_{ij} + \kappa_i \psi_i = - (1+\e)(3+\e)r^{-3-\e}+O(r^{-3-\tau}).
		\ee
with $\sigma_{ij} = O_2(r^{-\tau})$, $\kappa_i = O_1(r^{-1-\tau})$.
		Choosing $\e<\tau$, there is $R>0$ independent of $a$  so that $\Delta_g \psi \leq 0$ on $\R^3 \setminus B_R(0)$. Choose $a \gg 1$ so that $\psi > \phi$ on the complement of $\R^3 \setminus B_R(0)$ in $M$. We conclude, from the maximum principle, that $\phi \le \psi \leq a r^{-1}$ on $M$. Similarly, $\phi\ge -ar^{-1}$, giving the required bound on $|\phi|$. By the  interior Schauder estimates, \cite[Theorem 6.2]{GilbargTrudinger01} we have
\bee
		 |\phi|\le C|x|^{-1}, \  |\p\phi|\le C|x|^{-2},\ |\p\p \phi|\le C|x|^{-3}.
	\eee
	 	Extending $\phi$ to be a smooth function on in a compact set, we may assume that $\phi$ is defined in $\R^3$ so that $f=\Delta_0\phi=O(r^{-3-\tau}) $  by \eqref{e-laplacian-expansion} because $\Delta_g\phi=0$ near infinity. Let
		\[ w(x) := -\frac1{4\pi} \int_{\R^3} \frac1{|x-y|} f(y) \, dy. \]
		By the decay rate of $f$, this is well defined and $\Delta_0 w \equiv 0$.
		\begin{align} \label{eq-phi-expansion-1}
			\int_{\R^3} \frac{f(y)}{|x-y|} \, dy
				& = \underbrace{\int_{B_{r/2}(x)} \frac{f(y)}{|x-y|} \, dy}_{=: \, \mathbf{I}} + \underbrace{\int_{B_{r/2}(0)} \frac{f(y)}{|x-y|} \, dy}_{=: \, \mathbf{II}} \\
				& \qquad + \underbrace{\int_{\R^3 \setminus (B_{r/2}(x) \cup B_{r/2}(0))} \frac{f(y)}{|x-y|} \, dy}_{=: \, \mathbf{III}}. \nonumber
		\end{align}
		Since $|y| \geq r/2$ for $y \in B_{r/2}(x)$,
		\be \label{eq-phi-expansion-2}
			|\, \mathbf{I} \,| \le C r^{-3-\tau}\int_{B_{r/2}(x)} \frac{dy}{|x-y|} \le C r^{-1-\tau}.
		\ee
		Since outside $B_{r/2}(x)$, $|x-y| \ge \frac r2$. Hence
		\be \label{eq-phi-expansion-3}
			|\, \mathbf{III} \,| \le C r^{-1} \int_{\R^3 \setminus B_{r/2}(0)} |y|^{-3-\tau} \, dy \le C r^{-1-\tau}.
		\ee
		To estimate $\mathbf{II}$, let
		\[ \wt a := - \int_{\R^3} f(y) \, dy. \]
		Note that $\wt a$ is a finite number by the decay rate of $f$. For $r=|x|>1$,
		\be\label{eq-phi-expansion-4}
			\mathbf{II} + \frac{\wt a}{r} = \underbrace{-\frac1r\int_{\R^3 \setminus B_{r/2}(0)} f(y) \, dy}_{=: \, \mathbf{IV}} + \underbrace{\int_{B_{r/2}(0)} \lf( \frac1{|x-y|}-\frac1{|x|} \ri) f(y) \, dy}_{=: \, \mathbf{V}}.
		\ee
		By direct integration it is easy to see that
		\be \label{eq-phi-expansion-5}
			|\, \mathbf{IV} \,| \le C r^{-1-\tau}.
		\ee
		Since $|x-y|\ge \frac r2$ for $y\in B_{r/2}(0)$, we have
		\be \label{eq-phi-expansion-6}
			|\, \mathbf{V} \,| = \lf|\int_{B_{r/2}(0)}\lf(\frac{ 2x\cdot y-|y|^2}{|x|\,|x-y|\,(|x|+|x-y|)}\ri)f(y) \, dy \ri|  \le C r^{-1-\tau}.
		\ee
		By  \eqref{eq-phi-expansion-1}-\eqref{eq-phi-expansion-6}, we have
$w=ar^{-1}+O(r^{-1-\tau})$. Thus, by the maximum principle, $w \equiv \phi$.

	The higher order estimates, again, follow from \cite[Theorem 6.2]{GilbargTrudinger01}. This completes the proof of the lemma.
\end{proof}

It will often prove to be useful in this paper to invert coordinates on our asymptotic end. The inversion map is defined to be:
\be \label{e-inversion-map}
	\Phi : \R^3 \setminus \{0\} \to \R^3 \setminus \{0\}, \; y=\Phi(x) = |x|^{-2} x.
\ee
Using the inversion $y=x/|x|^2$, the point $\infty$   will correspond to the origin $y=0$. Suppose $g$ is a metric outside a compact set of the $x$-space in $\R^3$ and $g_{ij}=g(\p_{x^i},\p_{x^j})$ such that $g_{ij}=\delta_{ij}+\sigma_{ij}$. Using the above inversion, let  $h_{ij}=g(\p_{y^i},\p_{y^j})$ and $g_{ij}=g(\p_{x^i},\p_{x^j})$. Then
\be \label{e-inversion-map-1}
h_{ij} = |y|^{-4}\lf\{\delta_{ij}+ \sigma_{ij}+|y|^{-2}\lf[4y^iy^jy^ky^l\sigma_{kl}-2|y|^2
\lf(y^iy^k\sigma_{jk}+y^jy^k\sigma_{ik}\ri)\ri]\ri\}.
\ee
In particular, if $g_{ij}=\delta_{ij}$, then $h_{ij}=|y|^{-4}\delta_{ij}$.

\section{$\Fcirc(\Sigma, \gamma)$, $\Lcirc(\Sigma, \gamma)$ for disconnected $\Sigma$}
\label{s-f-l-disconnected}

We revisit $\F$, $\L$ from \cite[Definition 1.2]{M-M} where $\Sigma$ had been assumed connected for simplicity. We'll generalize results about $\F$, $\L$ from \cite{M-M} to match the revised Definition \ref{d-f-l} of the current paper.

The following two lemmas generalize verbatim---the connectedness of $\Sigma$ didn't play a role.

\begin{lemma}[Filling, cf. {\cite[Lemma 2.2]{M-M}}] \label{l-filling}
	Suppose $(\Omega, g) \in \Fcirc(\Sigma, \gamma)$ is such that:
	\begin{enumerate}
		\item $\Sigma_H = \p \Omega \setminus \Sigma_O$ is nonempty,
		\item $R(g) > 0$ on $\Sigma_H \subset \p \Omega$, and
		\item every component of $\Sigma_H$ is a stable minimal 2-sphere.
	\end{enumerate}
	Then for every $\eta > 0$, there exists $(D, h) \in \Fcirc(\Sigma, \gamma)$ with $H_h > H_g - \eta$ on its boundary $S_O$, which corresponds to $\Sigma_O$.
\end{lemma}

\begin{lemma}[Doubling, cf. {\cite[Lemma 2.3]{M-M}}] \label{l-doubling}
	Suppose $(\Omega, g) \in \Fcirc(\Sigma, \gamma)$, and that $\Sigma_H = \p \Omega \setminus \Sigma_O$ is nonempty. Let $D$ denote the doubling of $\Omega$ across $\Sigma_H$, so that $\p D = \Sigma_O \cup \Sigma_O'$, where $\Sigma_O'$ denotes the mirror image of $\Sigma_O$. For every $\eta > 0$ there exists a scalar-flat Riemannian metric $h$ on $D$ such that $(D, h)$ satisfies:
	\begin{enumerate}
		\item $\Sigma_O$ with the induced metric from $h$ is isometric to $(\Sigma, \gamma)$,
		\item $H_h > H_g$ on $\Sigma_O$, and
		\item $H_h' > H_g - \eta$ on $\Sigma_O'$,
	\end{enumerate}
	where $H_g$ denotes the mean curvature of $\Sigma_O$ in $(\Omega, g)$, and $H_h$, $H_h'$ denote the mean curvatures of $\Sigma_O$, $\Sigma_O'$ in $(D, h)$.
\end{lemma}

\begin{lemma}[cf. {\cite[Lemma 3.1]{M-M}}] \label{l-strict-inequality-l}
	If $(\Omega, g) \in \Fcirc(\Sigma, \gamma) \setminus \FF(\Sigma, \gamma)$, then
	\[ \frac{1}{8\pi} \int_{\p \Omega} H_g \, d\sigma < \LL(\Sigma, \gamma) \leq \Lcirc(\Sigma, \gamma). \]
\end{lemma}
\begin{proof}
	Use the doubling lemma (Lemma \ref{l-doubling}) to get $(D, h) \in \FF(\overline{\Sigma}, \overline{\gamma})$, where $\overline{\Sigma} = \Sigma \sqcup \Sigma'$, $\Sigma' \approx \Sigma$,  $\overline{\gamma}|_{\Sigma} = \gamma$, and which has
	\begin{equation} \label{e-strict-inequality-l}
		H_h > H_g \text{ on } \Sigma.
	\end{equation}
	It will be convenient to split up $(\Sigma, \gamma)$ into its connected components:
	\[ (\Sigma, \gamma) = (\Sigma_1, \gamma_1) \sqcup \cdots \sqcup (\Sigma_k, \gamma_k). \]
	Apply the cutting lemma \cite[Lemma 2.1]{M-M} (which didn't need to be generalized) $k$ times to $(D, h)$ to isolate the boundary components $(\Sigma_j, \gamma_j)$ and obtain $(\Omega_j, \gamma_j) \in \Fcirc(\Sigma_j, \gamma_j)$, $j = 1, \ldots, k$, with boundary mean curvatures $H_{h,j} = H_h|_{\Sigma_j} > H_g|_{\Sigma_j}$.	Fix $j = 1, \ldots, k$. Since $\Sigma_j$ is \emph{connected}, \cite[Proposition 5.1]{M-M} gives $\Lcirc(\Sigma_j, \gamma_j) = \LL(\Sigma_j, \gamma_j)$. Together with \eqref{e-strict-inequality-l}:
	\[ \LL(\Sigma_j, \gamma_j) = \Lcirc(\Sigma_j, \gamma_j) \geq \frac{1}{8\pi} \int_{\p \Omega_j} H_{h,j} \, d\sigma_j > \frac{1}{8\pi} \int_{\Sigma_j} H_g|_{\Sigma_j} \, d\sigma_j. \]
	This implies, by the additivity theorem (\cite[Theorem 1.2]{M-M}), that
	\[ \LL(\Sigma, \gamma) = \sum_{j=1}^k \LL(\Sigma_j, \gamma_j) > \sum_{j=1}^k \int_{\Sigma_j} H_g|_{\Sigma_j} \, d\sigma_j = \int_{\p \Omega} H_g \, d\sigma. \]
	The result follows.
\end{proof}

As a direct corollary we get:

\begin{prop}[cf. {\cite[Proposition 5.1]{M-M}}] \label{p-equality-l-ll}
	$\Lcirc(\Sigma, \gamma) = \LL(\Sigma, \gamma)$.
\end{prop}
\begin{proof}
	Take the supremum over $\Fcirc(\Sigma, \gamma) \setminus \FF(\Sigma, \gamma)$ in Lemma \ref{l-strict-inequality-l}.
\end{proof}

For the sake of completeness, let's also generalize some auxiliary results from \cite[Section 3]{M-M}, though they're not currently necessary.

\begin{prop}[cf. {\cite[Proposition 3.1]{M-M}}] \label{p-finiteness-l-spheres}
	If $\Sigma$ is a finite union of spheres and $\gamma$ is any Riemannian metric on $\Sigma$, then $\Lcirc(\Sigma, \gamma) < \infty$.
\end{prop}
\begin{proof}
	From Proposition \ref{p-equality-l-ll}, $\Lcirc(\Sigma, \gamma) = \LL(\Sigma, \gamma)$. The latter is known to be finite by combining the additivity theorem \cite[Theorem 1.2]{M-M} with the finiteness theorem for single 2-spheres \cite[Theorem 1.3]{M-M}.
\end{proof}

\begin{prop}[Rigidity in $\F$, cf. {\cite[Proposition 3.4]{M-M}}]
	If $(\Omega, g) \in \Fcirc(\Sigma, \gamma)$ attains the supremum $\Lcirc(\Sigma, \gamma)$, i.e., if
	\[ \frac{1}{8\pi} \int_{\p \Omega} H_g \, d\sigma = \Lcirc(\Sigma, \gamma), \]
	then $\Sigma$ is connected, $(\Omega, g) \in \FF(\Sigma, \gamma)$, and it's isometric to a mean-convex handlebody with flat interior whose genus is that of $\Sigma$. If $\Sigma$ were a sphere, then $(\Omega, g)$ can be isometrically immersed in $(\R^3, g_0)$.
\end{prop}
\begin{proof}
	The fact that $(\Omega, g) \in \FF(\Sigma, \gamma)$ follows from Lemma \ref{l-strict-inequality-l}. The result follows from the rigidity theorem for $\FF$, \cite[Theorem 1.4]{M-M}.
\end{proof}

\end{document}